\documentclass[12pt]{amsart}
\usepackage{mathrsfs}
\usepackage{amssymb}
\usepackage{cite}
\usepackage[all]{xy}
\usepackage{graphicx}
\usepackage{textcomp}
\usepackage{charter}
\usepackage[vflt]{floatflt}
\usepackage{mathtools}
\usepackage{enumerate}
\usepackage{wasysym}
\usepackage{rotating}
\usepackage{pdflscape}
\usepackage{comment}
\usepackage{fullpage}
\usepackage{bm}
\usepackage[pdftex,colorlinks]{hyperref}
\hypersetup{pdfauthor={Shamovich, E.},
bookmarksnumbered,
pdfstartview={FitH}}

\newtheorem{thm}{Theorem}[section]
\newtheorem{prop}[thm]{Proposition}
\newtheorem{lem}[thm]{Lemma}

\newtheorem*{thm*}{Theorem}
\newtheorem*{cor*}{Corollary}
\newtheorem{quest}[thm]{Question}
\newtheorem*{thmA}{Theorem A}
\newtheorem*{thmB}{Theorem B}
\newtheorem*{thmC}{Theorem C}

\theoremstyle{definition}
\newtheorem{dfn}[thm]{Definition}

\theoremstyle{remark}
\newtheorem{remark}[thm]{Remark}
\newtheorem{example}[thm]{Example}

\numberwithin{equation}{section}

\newcommand{\edim}{\mathrm{emb.dim}}
\newcommand{\affspan}{\mathrm{aff.span}}
\newcommand{\pick}{\mathrm{Pick}}

\usepackage{mathrsfs}
\newcommand{\Gr}{\mathrm{Gr}}

\newcommand{\GL}{\mathrm{GL}}

\renewcommand{\Re}{\mathrm{Re}}

\newcommand{\Aut}{\mathrm{Aut}}
\newcommand{\Span}{\mathrm{Span}}
\newcommand{\cE}{{\mathcal E}}
\newcommand{\cF}{{\mathcal F}}

\newcommand{\cH}{{\mathcal H}}

\newcommand{\cK}{{\mathcal K}}
\newcommand{\cL}{{\mathcal L}}
\newcommand{\cM}{{\mathcal M}}
\newcommand{\cN}{{\mathcal N}}

\newcommand{\cR}{{\mathcal R}}

\newcommand{\fM}{{\mathfrak M}}

\newcommand{\fU}{{\mathfrak U}}

\newcommand{\B}{{\mathbb B}}

\newcommand{\C}{{\mathbb C}}

\newcommand{\bP}{\mathbb{P}}

\newcommand{\N}{{\mathbb N}}

\newcommand{\D}{{\mathbb D}}

\newcommand{\bX}{{\mathbb X}}

\newcommand{\bH}{\mathbb{H}}

\begin{document}

\title{Deformations of complete Pick spaces}

\author{Prahllad Deb}
\address{Department of Mathematics, Indraprastha Institute of Information Technology Delhi}
%\curraddr{Department of Mathematics, IIIT - Delhi}
%\email{prahllad@iiitd.ac.in}

\author{Jonathan Nureliyan}
\address{Department of Mathematics, Ben-Gurion University of the Negev}
%\email{nuriejoh@post.bgu.ac.il}

\author{Eli Shamovich}
\address{Department of Mathematics, Ben-Gurion University of the Negev}
%\email{shamovic@bgu.ac.il}

\thanks{PD was partially supported by the Post-doctoral fellowship at Ben-Gurion University of the Negev. ES was partially supported by BSF Grant number 2022235.}

\begin{abstract}
Motivated by the work of Pandey, Ofek, and Shalit on the one hand and deformation theory on the other, we study the Grassmannian of $n$-dimensional multiplier-coinvariant subspaces of the Drury-Arveson space. We show that this space admits a natural map to the symmetrized polyball that induces an isomorphism between the configuration space of $n$ points in the ball and the subspace of projection onto spaces spanned by $n$ distinct kernels. We discuss the tautological bundle on our Grassmannian and the corresponding operator algebra bundle. We construct examples of bundles of complete Pick spaces from homogeneous hypersurfaces in $\B_d$. Along with these bundles, we construct examples of Cowen-Doulas tuples of operators from the compressed Arveson $d$-shift.
\end{abstract}

\maketitle

\section{Introduction}

Let $\D \subset \C$ be the open unit disc. The classical Pick interpolation problem asks to find an analytic function $f \colon \D \to \overline{\D}$ that satisfies $f(z_i) = w_i$, for some $z_1,\ldots,z_n \in \D$ and $w_1,\ldots,w_n \in \C$ that were initially prescribed. This problem admits an elegant solution due to Pick and Nevanlinna. Namely, such a function exists if and only if the Pick matrix $\left(\frac{1 - w_i \overline{w_j}}{1 - z_i \overline{z_j}}\right)_{i,j=1}^n$ is positive. Moreover, the problem for matrix-valued functions admits essentially the same solution. This fact can be encoded in the statement that the reproducing kernel Hilbert space $H^2(\D)$ of all analytic functions on $\D$ with square-summable Taylor coefficients at the origin is a complete Pick space. In this paper, we assume that the reader has some familiarity with reproducing kernel Hilbert spaces (RKHS for short) and the complete Pick property. We give a brief description of the complete Pick property in Subsection \ref{subsec:DA}. One of the standard references on these topics is the excellent book \cite{AglerMcCarthy-book}.

The complete Pick property is a powerful geometric property. The strength of the complete Pick property manifests itself in several ways. Suppose $\cH$ is an RKHS on a set $X$ with the complete Pick property. In that case, a celebrated theorem of Agler and McCarthy  \cite{AglerMcCarthy-completeNP} guarantees that there exists a $d \in \N \cup \{\infty\}$ and an embedding $b \colon X \to \B_d$, where $\B_d$ is the unit ball of $\C^d$ (we use the convention that $\C^{\infty}$ is $\ell^2$). Moreover, $\cH$ is isomorphic as an RKHS to a space of analytic functions on a variety in $\B_d$. The family of universal complete Pick spaces of analytic functions on $\B_d$ is the family $H^2_d$ where $H^2_d$ is the Drury-Arveson space on $d$ variables. The isomorphism of $\cH$ with a subspace of the Drury-Arveson space induces a completely isometric isomorphism of the corresponding multiplier algebras.  This observation led Davidson, Ramsey, and Shalit \cite{DRS1, DRS2} to study isomorphisms of multiplier algebras of subvarieties of $\B_d$ cut out by WOT closed ideals of multiplier on $H^2_d$. One of their main results is that two such algebras are completely isometrically isomorphic if and only if there exists an automorphism of $\B_d$ that maps one variety onto the other. A similar result for finite-dimensional Pick spaces was obtained by Rochberg \cite{Rochberg-hyp_geom}.

The importance of the complete Pick property is not restricted to the classical function spaces. Recently, the complete Pick property was extended to the noncommutative setting \cite{AglerMcCarthy-free_pick,AriasPopescu-interpolation, BMV-ncrkhs, BMV-int,DavidsonPitts-NPint}. In particular, the full Fock space is a complete Pick nc-RKHS. In many ways, the full Fock space is a better analog of $H^2(\D)$ than $H^2_d$. Moreover, $H^2_d$ can be identified with the space of symmetric tensors in the full Fock space. This fact is utilized in this paper, and we discuss the full Fock space in Subsection \ref{subsec:Blaschke}.

The inspiration for this paper came from the work of Ofek, Pandey, and Shalit \cite{OPS}. In \cite{OPS}, the authors construct the space of $\pick_n$ isomorphism classes of complete Pick spaces on $n$ points, where $n \in \N$ is fixed. They equip the space with a Banach-Mazur type distance and other distances that induce the same topology on the space. Our approach to this, however, is influenced by algebraic geometry. By the aforementioned theorem of Agler and McCarthy, we know that if $\cH$ is a complete Pick space on $X$, a set of cardinality $n$, then there exists $d \leq n-1$, such that $X$ embeds into $\B_d$, where $\B_d$ is the unit ball of $\C^d$, and $\cH$ is isomorphic to a multiplier-coinvariant subspace of the Drury-Arveson space $H^2_d$. Therefore, we consider the Grassmannian $\bP_{d,n}$ of $n$-dimensional multiplier-coinvariant subspaces of $H^2_d$. On $\bP_{d,n}$, the strong and norm topologies coincide, so we consider this space equipped with the metric induced by the norm. There is a natural action of the group of holomorphic automorphisms of $\B_d$ on $\bP_{d,n}$. The first main result of the paper is:
\begin{thmA}[Theorems \ref{thm:proj_continuous} and \ref{thm:inverse_map}]
There is a natural continuous map $\Psi_{d,n} \colon \bP_{d,n} \to \bX_{d,n}$. Here $\bX_{d,n}$ is the symmetrized polyball equipped with the symmetric distance (see Subsection \ref{subsec:ball} for details). Moreover, $\Psi_{d,n}$ restricts to a homeomorphism from $\bP_{d,n}^0$, the subspace of projections in $B(H^2_d)$ onto subspaces spanned by $n$ distinct kernels to the configuration space of $n$ unordered points in $\B_d$ that we denote by $\bX_{d,n}^0$.
\end{thmA}
As seen in Section \ref{sec:grassmannian}, the space $\bP_{d,n}$ behaves like a blowup of $\bX_{d,n}$ (see for example \cite{FultonMacPherson}). At this point, it is natural to ask what is the advantage of this point of view compared to the space of isomorphism classes of complete Pick spaces on $n$ points. One key difference is that on $\bP_{d,n}$, we have the tautological bundle over this Grassmannian, and the bundle is equivariant with respect to the action of $\Aut(\B_d)$. We discuss bundles of complete Pick spaces in Section \ref{sec:bundles}. In particular, we raise Question \ref{quest:universality}. Roughly speaking, the question is whether the Grassmannian of complete Pick spaces $\bP_{d,n}$ possesses the universal property of the Grassmannian to some extent. Classically, given a rank $n$ vector bundle $\cE$ on a compact topological space $Z$, there exists a $d$ and a map $\varphi \colon Z \to \Gr(n,d)$, such that $\cE$ is the pullback via $\varphi$ of the tautological bundle on the Grassmannian. We ask whether this is true for bundles of complete Pick spaces.

We use a well-known construction of finite morphisms from algebraic geometry to provide examples of bundles of complete Pick spaces. Let $p$ be a homogenous polynomial that does not vanish on the $d$-th coordinate axis. Let $V(p) \subset \B_d$ be the hypersurface cut out by $p$. We project $V(p)$ onto the first $d-1$ coordinates and obtain a neighborhood of the origin $U \subset \B_{d-1}$ and a vector bundle $\cE_p$ on $U$. The fiber $\cE_{p,\lambda}$ for $\lambda \in U$ can be described as $\bigcap_{j=1}^{d-1} \ker(S_j^* - \overline{\lambda_j})$, where $S_1,\ldots,S_{d-1}$ are the compressions of the first $d-1$ coordinates of the Arveson $d$-shift to the complement of the ideal generated by $p$. This result is obtained from the following theorem.
\begin{thmB}
Let $p \in \C[z_1,\ldots,z_d]$ be a homogeneous polynomial, such that $p$ does not vanish on the $d$-th coordinate axis. Let $\cH_p = H^2_d \ominus M_p H^2_d$ and $S_j = P_{\cH_p} M_{z_j}|_{\cH_P}$ be the compressions of the Arveson $d$-shift. Then, there exists a neighborhood of the origin $U \subset \B_{d-1}$, such that the $d-1$-tuple $S_1^*,\ldots,S_{d-1}^*$ is in the Cowen-Douglas class on $U$.
\end{thmB}
The Cowen-Douglas class of operators was introduced in the seminal work \cite{CowenDouglas=CGOT} as a vast generalization of the backward shift on $H^2(\D)$. Tuples of commuting operators in the Cowen-Douglas class were studied by Curto and Salinas \cite{CurtoSalinas} and many others. In particular, Lu, Yang, and Zu \cite{LuYangZu} have studied operators in the Cowen-Douglas class arising from compressions of the shifts on the Hardy space of the bidisc. Their results are akin to ours but in a different setting.

In Section \ref{sec:stratification}, we discuss a stratification of $\bX_{d,n}^0$ and $\pick_n$. In particular, we show the following theorem.

\begin{thmC}[Theorem \ref{thm:reg_pick_spaces}]
There is an open dense subset $\pick_n^{reg} \subset \pick_n$ that has a structure of a smooth manifold. Moreover, the restriction of the tautological bundle on $\bP_{d,n}$ to the set of regular points is a vector bundle on $\pick_n^{reg}$.
\end{thmC}

\section*{Acknowledgements}

The authors thank Matt Kennedy for introducing us to the noncommutative approach to the McCullough-Trent theorem. We also thank Orr Shalit and Victor Vinnikov for their helpful comments and discussions.

\section{Preliminaries} \label{sec:background}

\subsection{Geometry of the unit ball and its symmetrization} \label{subsec:ball}

Let $d \in \N$. We will denote the open unit ball in $\C^d$ by
\[
\B_d = \left\{ z \in \C^d \mid \sum_{j=1}^d |z_j|^2 < 1 \right\}.
\]
In particular, if $d = 1$, $\B_1 = \D$, the open unit disc. This section collects some necessary information from \cite{Rudin-book}. The set $\B_d$ is a homogenous space. For every point $0 \neq \lambda \in \B^d$, there is a unique involution that exchanges $\lambda$ and $0$ given by
\[
\varphi_{\lambda}(z) = \dfrac{\lambda - P_{\lambda} z - \sqrt{1 - \|\lambda\|^2} (I - P_{\lambda}) z}{1 - \langle z, \lambda \rangle}.
\]
Here, $P_{\lambda} z = \frac{1}{\|\lambda\|^2} \langle z, \lambda \rangle \lambda$. The automorphism group of the ball is generated by the involutions described above and unitaries. We will denote the automorphism group by $\Aut(\B_d)$. In fact, $\Aut(\B_d)$ is a quotient of $\mathrm{SU}(1,d)$ by a finite subgroup.

The pseudohyperbolic distance on $\B_d$ is defined by $d_{ph}(z,w) = \|\varphi_z(w)\|$. The analog of the Schwarz-Pick lemma for the ball tells us that every analytic self-map of $\B_d$ is a contraction with respect to $d_{ph}$. Hence, automorphisms are isometries with respect to $d_{ph}$. 

In this paper, we will deal with finite subsets of $\B_d$, allowing for multiplicities. To this end we consider the polyball $\B_d^n = \B_d \times \cdots \times \B_d$ ($ n $ times). The symmetric group $S_n$ acts on this product by permuting the components. We will denote by $\bX_{d,n}$ the quotient by this action. We call $\bX_{d,n}$ the symmetrized polyball (also known as the $n$-th symmetric power of $\B_d$). If we consider the polyball as a subset of $\C^{dn}$, then $\bX_{d,n}$ can be viewed as an image in $\C^{dn}$ under a symmetrization map. In particular, if $d=1$, our $\bX_{1,n}$ is biholomorphic to the symmetrized polydisc (see for example the works of Agler, Young, and Lykova in the case of the bidisc \cite{AglerYoung-hyp_geom_bidisc,AglerYoung-comp_geod_bidisc,ALY-ext_hol_bidisc}). It is sometimes convenient to consider points $ [X] \in \bX_{d,n}$ as formal sums. We consider the free abelian group $D(\B_d)$ generated by the points of $\B_d$. Namely, the elements of $D(\B_d)$ are finite formal sums of points in $\B_d$ with integer coefficients, for example $2[(0,0)] + [(1/2,0)] - [(1/2 i, - 1/2 i)]$. The positive elements $D_{+}(\B_d) \subset D(\B_d)$ are all such combinations with non-negative coefficients. The degree map is defined as $\deg \left( \sum_{j=1}^n m_j [\lambda_j] \right) = \sum_{j=1}^n m_j$. We note that we have the following identification
\[
\bX_{d,n} = \left\{ D \in D_{+}(\B_d) \mid \deg D =n \right\}.
\]
It will be sometimes convenient to view elements of $\bX_{d,n}$ as multisets and sometimes as described above.

The space $\bX_{d,n}$ inherits a metric from the polyball. First, we observe that $$d(X,Y) = \max_{1 \leq j \leq n} d_{ph}(x_j, y_j)$$ defines a metric on $ \B_d^n $. Here, we write $X = (x_1, \ldots, x_n), Y = (y_1, \ldots, y_n) \in \B_d^n $. In \cite{OPS}, the symmetric distance between the images $[X]$ and $[Y]$ of $X$ and $Y$, respectively, in $\bX_{d,n}$ is given by
\[
d_s([X],[Y]) = \min_{\sigma \in S_n} d(X, \sigma(Y)).
\]
Another distance that appears in various contexts that is related to the above is the optimal matching distance on $\C^{dn}$ defined by
\[
d_o([X],[Y]) = \min_{\sigma \in S_n} \max_{1 \leq j \leq n} \|x_j - y_{\sigma(j)}\|.
\]

\begin{lem} \label{lem:opt_vs_sym}
Let $0 < r < 1$, then for every two subsets of $n$ points $X, Y \subset r \overline{\B_d}$,
\[
d_s([X],[Y]) \leq \frac{1}{1 - r^2} d_o([X],[Y]).
\]
\end{lem}
\begin{proof}
For $ z, w \in \B_d^n $, it follows from \cite[Theorem 2.2.2]{Rudin-book} that 
$$ \|\varphi_w(z)\|^2 = 1 - \frac{(1 - \|w\|^2)(1 - \|z\|^2)}{|1 - \langle z, w \rangle|^2} = \frac{|1 - \langle z, w \rangle|^2 - (1 - \|w\|^2)(1 - \|z\|^2)}{|1 - \langle z, w \rangle|^2} . $$
Consequently, from the definition of the pseudohyperbolic distance on $ \B_d^n $, we have that 
\begin{eqnarray*}
d_{ph} ( z, w )^2 & = & \frac{|1 - \langle z, w \rangle|^2 - (1 - \|w\|^2)(1 - \|z\|^2)}{|1 - \langle z, w \rangle|^2} \\
& = & \frac{\|w\|^2 + \|z\|^2 - \langle z, w \rangle - \langle w, z \rangle + | \langle z, w \rangle|^2  -\|z\|^2 \|w\|^2}{|1 - \langle z, w \rangle|^2}  \\
& = & \frac{\|z - w\|^2 + | \langle z, w \rangle|^2  -\|z\|^2 \|w\|^2}{|1 - \langle z, w \rangle|^2} \\
& \leq & \frac{1}{(1-r^2)^2} \|z-w\|^2
\end{eqnarray*} 
verifying the desired identity.
\begin{comment}
{\color{magenta} I have realigned this proof without changing the argument. We may keep one of these two versions depending upon your choice.}

By \cite[Theorem 2.2.2]{Rudin-book}, we have that for all $z, w \in \B_d$,
\begin{multline*}
d_{ph}(z,w)^2 = \|\varphi_w(z)\|^2 = 1 - \frac{(1 - \|w\|^2)(1 - \|z\|^2)}{|1 - \langle z, w \rangle|^2} = \frac{|1 - \langle z, w \rangle|^2 - (1 - \|w\|^2)(1 - \|z\|^2)}{|1 - \langle z, w \rangle|^2}  \\
= \frac{\|w\|^2 + \|z\|^2 - \langle z, w \rangle - \langle w, z \rangle + | \langle z, w \rangle|^2  -\|z\|^2 \|w\|^2}{|1 - \langle z, w \rangle|^2}  \\ = \frac{\|z - w\|^2 + | \langle z, w \rangle|^2  -\|z\|^2 \|w\|^2}{|1 - \langle z, w \rangle|^2} \leq \frac{1}{(1-r^2)^2} \|z-w\|^2.
\end{multline*}
\end{comment}
\end{proof}

Note that since $\B_d$ is complete with respect to the pseudohyperbolic metric, so is $\B_d^n$ with respect to $d$. 

\begin{prop}
The metric space $(\bX_{d,n},d_s)$ is complete
\end{prop}
\begin{proof}
Let $ \{ X_k \} \in \B_d^n$ be a sequence, such that the corresponding sequence $ \{ [X_k] \} \subset \bX_{d,n}$ is Cauchy with respect to $d_s$. For each $\sigma \in S_n$, we set
\[
A_{\sigma} = \left\{ k \in \N \mid d_s([X_k],[X_{k+1}]) = d(X_k,\sigma(X_{k+1})\right\}
\]
Since $\N = \bigcup_{\sigma \in S_n} A_{\sigma}$ at least one of the $A_{\sigma}$ is infinite. Passing to a subsequence, we may assume that there is always one $\sigma$. Now set $Y_k = \sigma^{k-1}(X_k)$ and note that
\[
d(Y_k,Y_{k+1}) = d(\sigma^{k-1}(X_k),\sigma^k(X_{k+1})) = d(X_k, \sigma(X_{k+1})) = d_s([X_k], [X_{k+1}]).
\]
Hence, $ \{ Y_k \} $ is a Cauchy sequence that converges to some $Y$. By continuity of the quotient map $ \{ [X_k] \} $ converges to $[Y]$.
\end{proof}

\subsection{Joint spectrum of commuting tuples} \label{subsec:joint_spectrum}

Let $A = (A_1, \ldots, A_d) \in B(\cH)^d$ be a commuting $d$-tuple of operators. For a point $\lambda \in \C^d$, we will write $A -\lambda I$ for the $d$-tuple $(A_1 - \lambda_1 I, \ldots, A_d - \lambda_d I)$. For a multi-index $\alpha \in \left(\N \cup \{0\}\right)^d$ we will write $A^{\alpha} = \prod_{j=1}^d A_j^{\alpha_j}$. The row and column norms of a $d$-tuple are defined by
\[
\|A\|_{row} = \left\| \sum_{j=1}^d A_j A_j^*\right\|^{1/2} \text{ and } \|A\|_{col} = \left\| \sum_{j=1}^d A_j^* A_j\right\|^{1/2}.
\]
In particular, we say that $A$ is a row contraction, if $\|A\|_{row} \leq 1$. If the inequality is strict, $A$ is called a strict row contraction. We say that $A$ is a commuting (strict) row contraction if $A$ is both a commuting tuple and a (strict) row contraction. There is an optimal inequality between these norms for every $d$-tuple of matrices $A$:
\[
\frac{1}{\sqrt{d}} \|A\|_{col} \leq \|A\|_{row} \leq \sqrt{d} \|A\|_{col},
\]

The notion of a joint spectrum of a commuting tuple of matrices is well-established. In general, for a $d$-tuple of commuting operators, there are many notions of a joint spectrum that are not equivalent (see \cite{Curto-spectrum} for an excellent overview). Fortunately for us, all of these definitions coincide in the case of matrices and agree with the following one.

\begin{dfn} \label{dfn:joint_spectrum}
We define the joint spectrum of a commuting tuple $A \in M_n^d$ to be
\[
\sigma(A) = \left\{ \lambda \in \C^d \mid \bigcap_{j=1}^d \ker (A_j - \lambda_j I) \neq \{0\} \right\}.
\]
\end{dfn}
We note that it is easy to see using the simultaneous upper-triangular form that for a commuting row contraction, we have that $\sigma(A) \subset \overline{\B_d}$. It also follows from the work of Popescu \cite{Popescu-similarity} (see also \cite{SSS2} for a different proof) that $\sigma(A) \subset \B_d$ if and only if $A$ is jointly similar to a strict row contraction. In other words, there exists $S \in \GL_n$, such that the $d$-tuple $S^{-1} A S = (S^{-1} A_1 S, \ldots, S^{-1} A_d S)$ is a strict row contraction.

In this paper, we will need more information from the spectrum. To this end, we adjust the definition slightly. First, we need the notion of a root subspace of a tuple $A$.
\[
R_{\lambda}(A) = \bigcap_{|\alpha| = n} \ker(A - \lambda I)^{\alpha}.
\]
We note that in the case $d=1$, $R_{\lambda}$ is the classical root subspace of $A$. Namely, the space spanned by the eigenvectors corresponding to $\lambda$ and the associated vectors, if any. It follows from \cite[Corollary 9.1.3]{GLR-invaraint_book} that $\C^n = \sum_{\lambda \in \sigma(A)} R_{\lambda}(A)$. This sum is a non-orthogonal direct sum decomposition. In particular, $\sum_{\lambda \in \sigma(A)} \dim R_{\lambda}(A) = n$. 

\begin{dfn} \label{dfn_joint_spectrum_with_multiplicity}
Let $A \in M_n^d$ be similar to a strict commuting row contraction. We set
\[
\sigma_J(A) = \sum_{\lambda \in \sigma(A)} \dim R_{\lambda}(A) [\lambda] \in \bX_{d,n}. 
\]
\end{dfn}

It will be important for us in the later stages of the paper to understand the continuity properties of the joint spectrum. In this presentation, we follow Elsner \cite{Elsner-joint}. We fix $n \in \N$ and $A, B \in M_n^d$ two commuting $d$-tuples. We define $S_n(\Delta,r)$ as the spectral radius of the matrix
\[
\begin{pmatrix}
0 & \Delta & &  \\ & \ddots & &  \\ & & 0 & \Delta \\ r & \cdots & r & r
\end{pmatrix}.
\]
The number $\tilde{\Delta}(B)$ measures the non-normality of $B$. More precisely, if $D_n$ stands for the algebra of diagonal matrices and $N_n$ for the algebra of strictly upper-triangular ones, then one defines
\[
\tilde{\Delta}(B) = \min \{ \|N\|_{col} \mid U^* B U = D + N,\, U \in U_n,\, D \in D_n, \text{ and } N \in N_n\}
\]
By \cite[Theorem 7.1]{Elsner-joint}, we have that
\[
d_H(\sigma(A),\sigma(B)) \leq S_n(\tilde{\Delta}(B), \|A - B\|_{col}).
\]
Here $d_H$ stands for the Hausdorff distance. Now as observed in \cite{Elsner-var} and used in \cite{Bhatx2} to the joint spectrum case,
\[
d_H(\sigma(A), \sigma(B)) \leq n^{1/n} (2M)^{1-1/n}\|A-B\|_{col}^{1/n}.
\]
Here $M = \max\{\|A\|_{col}, \|B\|_{col}\}$. If we replace the column norm with the row norm, we obtain the following lemma.

\begin{lem} \label{lem:bound_on_hausdorff}
Let $A, B \in M_n^d$ be two commuting row contractions, then there exists a constant $c(n,d) > 0$ that depends only on $n$ and $d$, such that
\[
d_H(\sigma(A), \sigma(B)) \leq c(n,d) \|A - B\|_{row}^{1/n}.
\]
\end{lem}
However, the Hausdorff distance is unsatisfactory for our purposes since it ignores multiplicities. Therefore, we need to refine this result. To this end, we need some lemmas. First, Theorem 3.4 in \cite{KosPles} together with the definition of a stable subspace of a commuting $d$-tuple $A$ yield the lemma below.
\begin{lem} \label{lem:approx_root_subspaces}
Let $A \in M_n^d$ be a commuting tuple of matrices. Then for every $\lambda \in \sigma(A)$ and for every $\varepsilon > 0$, there exists $\delta > 0$, such that if $B \in M_n^d$ is a commuting tuple, such that $\|A - B\| < \delta$, then there exists a $B$-invaraint subspace $\cN \subset \C^n$, such that $\|P_{\cN} - P_{\cR_{\lambda}(A)}\| < \varepsilon$.
\end{lem}
The next lemma is a consequence of \cite[Theorem VI.1.8]{Bhatia-book}.
\begin{lem} \label{lem:dist_isom}
Let $\cH$ be a Hilbert space and $\cK_1, \cK_2 \subset \cH$ be $n$-dimensional subspaces. Let $P_1, P_2 \in B(\cH)$ be the projections onto $\cK_1$ and $\cK_2$, respectively. Given, an isometry $V_1 \colon \C^n \to \cH$, such that $\mathrm{ran} V_1 = \cK_1$, we can find an isometry $V_2 \colon \C^n \to \cH$, such that $\mathrm{ran} V_2 = \cK_2$ and, moreover,
\[
\|P_1 - P_2\| \leq \|V_1 - V_2\| \leq \sqrt{2} \|P_1 - P_2\|.
\]
\end{lem}
Now, we are ready to prove the proposition that we will require. Recall that Definition \ref{dfn_joint_spectrum_with_multiplicity} describes $\sigma_J(A) \in \bX_{d,n}$ as the joint spectrum of $A$ with keeping track of multiplicities.
\begin{prop} \label{prop:convergence_of_spectra}
Let $A_m \in M_n^d$ be a sequence of commuting row contractions that converges to a commuting row contraction $A$, then
\[
\lim_{n \to \infty} d_o(\sigma_J(A_n), \sigma_J(A)) = 0.
\]
\end{prop}
\begin{proof}
Let us fix $0 < \varepsilon < 1$. We obtain $\delta > 0$ from Lemma \ref{lem:approx_root_subspaces} and let $m_0 \in \N$ be such that for $m \geq m_0$, we have that $\|A_m - A\| < \min\{\delta, \varepsilon\}$. Now fix $\lambda \in \sigma(A)$, then for each $m \geq m_0$, we can find $\cN_m \subset \C^n$, that is $A_m$-invariant and $\|P_{\cN_m} - P_{\cR_{\lambda}(A)}\| < \varepsilon$. In particular, $\dim \cN_m = \dim \cR_{\lambda}(A) = k$. Let $V, W \colon \cK \to \C^n$ be isometries with $\mathrm{ran} V = \cN_m$ and $\mathrm{ran} W = \cR_{\lambda}(A)$, and $\|V - W\| < \sqrt{2} \varepsilon$. Then,
\[
\|V^* A_m V - W^* A W\| \leq \|A_m V - A W\| + \|V - W\| < (1 + 2 \sqrt{2}) \varepsilon.
\]
Note that since these spaces are invariant, $\sigma(W^* A W) = \{\lambda\}$ and $\sigma(V^* A_m V) \subset \sigma(A)$. In particular, we have that $\sigma(V^* A_m V)$ is contained in a disc of radius $c(n,d)(1 + 2 \sqrt{2})^{1/n} \varepsilon^{1/n}$ centered at $\lambda$. Now, let us take $\varepsilon$ small enough so that the discs with the above radii centered at different points of the spectrum are disjoint. We get that each such disc contains precisely $\dim \cR_{\lambda}(A)$ joint eigenvalues of $A_m$ counting multiplicities. Therefore, we can choose a permutation that shows that
\[
d_o(\sigma_J(A_m), \sigma_J(A)) \leq c(n,d)(1 + 2 \sqrt{2})^{1/n} \varepsilon^{1/n}.
\]
This completes the proof.
\end{proof}

\subsection{Drury-Arveson space} \label{subsec:DA}

The Drury-Arveson space is a Hilbert space of functions first considered by Drury \cite{Drury}. However, Arveson \cite{Arv-subalg3} was the first to identify its many properties analogous to the classical Hardy space $H^2(\D)$. We will provide only necessary information on the Drury-Arveson space in this section. The interested reader should consult the outstanding surveys \cite{Hartz-DA} and \cite{Shalit-DA}. We will assume basic familiarity with reproducing kernel Hilbert spaces (RKHS for short). We refer the reader to \cite{AglerMcCarthy-book} for the necessary background.

The Drury-Arveson space on $d$ variables $H^2_d$ (also known as the symmetric Fock space) is the completion of the polynomial ring $\C[z_1,\ldots,z_d]$ with respect to the following inner product
\[
\langle z^{\alpha}, z^{\beta} \rangle = \delta_{\alpha, \beta} \frac{\alpha!}{|\alpha|!}.
\]
Here, $\alpha, \beta \in \left(\N \cup \{0\}\right)^d$ are multi-indices. We write $z^{\alpha} = \prod_{j=1}^d z_j^{\alpha_j}$, $\alpha! = \prod_{j=1}^d \alpha_j !$, and $|\alpha| = \sum_{j=1}^d \alpha_j$. This inner product is also known as the Bombieri inner product \cite{BBEM}. We observe, that if $d=1$, that $H^2_1 = H^2(\D)$. It was observed by Arveson \cite{Arv-subalg3} that $H^2_d$ is a reproducing kernel Hilbert space with the kernel
\[
k(z,w) = \frac{1}{1 - \langle z, w \rangle};
\]
In particular, $H^2_d$ is a space of holomorphic functions on $\B_d$. For $\lambda \in \B_d$, we will denote by $k_{\lambda} \in H^2_d$ the kernel function $k_{\lambda}(z) = k(z,\lambda)$. More generally, for a multi-index $\alpha \in \left(\N \cup \{0\}\right)^d$, we will denote by $k_{\lambda}^{(\alpha)} \in H^2_d$ the vector such that for each $f \in H^2_d$,
\[
\langle f, k^{(\alpha)}_{\lambda} \rangle = \dfrac{\partial^{\alpha} f}{\partial z^{\alpha}}(\lambda).
\]
Now, by the properties of the inner product on $H^2_d$ and the Taylor series expansion at the origin, we have that
\[
\langle f, z^{\alpha} \rangle = \dfrac{1}{\alpha!}\dfrac{\partial^{\alpha} f}{\partial z^{\alpha}}(0) \|z^{\alpha}\|^2 = \dfrac{1}{|\alpha|!} \dfrac{\partial^{\alpha} f}{\partial z^{\alpha}}(0).
\]
Hence, $k_0^{(\alpha)} = |\alpha|! z^{\alpha}$. More generally, we have the following lemma:

\begin{lem} \label{lem:derivatives}
Let $\lambda \in \B_d$ and let $\alpha \in (\N \cup \{0\})^n$ be a multi-index. Then, $$k_{\lambda}^{(\alpha)}(z) = \dfrac{\partial^{\alpha} k}{\partial \bar{w}^{\alpha}} (z,\lambda) = \dfrac{|\alpha|! z^{\alpha}}{(1 - \langle z, \lambda \rangle)^{n+1}}$$
\end{lem}
\begin{proof}
We prove the claim by induction on $|\alpha| = n$. The claim for $n = 0$ is the definition of the kernel of $H^2_d$. Hence, we may assume that $n > 0$. Let us assume without loss of generality that $\alpha_1 > 0$. Then $\alpha = \alpha' + e_1$. Let us write now
\[
\dfrac{\partial^{\alpha} f}{\partial z^{\alpha}}(\lambda) = \lim_{z \to 0} \dfrac{1}{z} \left[\dfrac{\partial^{\alpha'} f}{\partial z^{\alpha'}}(\lambda + z e_1) - \dfrac{\partial^{\alpha'} f}{\partial z^{\alpha'}}(\lambda)\right] = \lim_{z \to 0} \left\langle f, \frac{ k_{\lambda + z e_1}^{(\alpha')} - k_{\lambda}^{(\alpha')} }{\bar{z}} \right\rangle.
\]
In other words, $\frac{1}{\bar{z}} \left(k_{\lambda + z e_1}^{(\alpha')} - k_{\lambda}^{(\alpha')}\right)$ converges weakly to $k_{\lambda}^{(\alpha)}$. Now fix any $\mu \in \B_d$ and we have by induction
\[
k_{\lambda}^{\alpha}(\mu) = \lim_{z \to 0}\dfrac{1}{\bar{z}} \left[\dfrac{\partial^{\alpha'} k}{\partial \bar{w}^{\alpha'}}(\mu, \lambda + z e_1) - \dfrac{\partial^{\alpha'} k}{\partial \bar{w}^{\alpha'}}(\mu, \lambda) \right].
\]
The proof is complete.
\end{proof}

It is not hard to check that the multiplication by coordinates defines contractive operators $M_{z_j} f = z_j f$, for $1 \leq j \leq d$ on $H^2_d$. Moreover, the row $M_z = (M_{z_1},\ldots,M_{z_d})$ is a contraction. A row contraction $T = (T_1,\ldots,T_d) \in B(\cH)^d$ is called pure, if for every $\xi \in \cH$, $\lim_{n\to \infty} \sum_{|\alpha|=n} \| T^{\alpha *}\xi\|^2 =0$. Arveson proved that $M_z$ is the universal pure row contraction. This property is one important universality property of the Drury-Arveson space. We remark in passing that a finite-dimensional pure row contraction is similar to a strict contraction see \cite{Popescu-similarity} and \cite{SSS2}. Therefore, the joint spectrum of a finite-dimensional pure row contraction is contained in $\B_d$.

The multiplier algebra of $H^2_d$ will be denoted by $\cM_d$. This is an operator algebra equipped with the multiplier norm. It is well known that $\cM_d$ is the WOT closure of the algebra generated by the $M_{z_j}$. By standard facts about reproducing kernel Hilbert spaces, the algebra $\cM_d$ consists of bounded analytic functions on $\B_d$. However, the inclusion $\cM_d \hookrightarrow H^{\infty}(\B_d)$ is contractive and strict, unless $d = 1$. The second important universal property of the Drury-Arveson space was proved by Agler and McCarthy \cite{AglerMcCarthy-completeNP} based on earlier works of McCullough and of Quiggin. In order to describe this universal property, we need first to recall the complete Pick property.

Recall that the classical Pick interpolation problem is to find an analytic function $f \colon \D \to \overline{\D}$ satisfying $f(\lambda_i) = w_i$, for $1 \leq i \leq n$, where $\lambda_1,\ldots,\lambda_n \in \D$ and $w_1,\ldots,w_n \in \overline{\D}$. The solution of Pick and of Nevanlinna was that such a function exists if and only if the matrix  $\left(\frac{1 - w_i \overline{w_j}}{1 - \lambda_i \overline{\lambda_j}}\right)_{i,j=1}^n$ is positive -- this matrix is known as the Pick matrix. Moreover, it is true that if we want to find an analytic matrix-valued function $f$ of sup-norm at most $1$ that satisfies $f(\lambda_i) = W_i$, then the condition is essentially unchanged. Namely, we only need the positivity of the block matrix $\left(\frac{I - W_i W_j^*}{1 - \lambda_i \overline{\lambda_j}}\right)_{i,j=1}^n$. Note that the function $k(z,w) = \frac{1}{1 - z \bar{w}}$ is the Szegö kernel, which is the reproducing kernel for $H^2(\D)$. This leads to a generalization of this question to arbitrary reproducing kernel Hilbert spaces. To be more precise, we say that a reproducing kernel Hilbert space $\cH_k$ of functions on a set $X$ has the scalar Pick property, if for every set of points $x_1,\ldots,x_n \in X$ and values $w_1,\ldots,w_n \in \overline{\D}$, there exists a contractive multiplier $f$ on $\cH_k$, such that $f(x_i) = w_i$ if and only if the Pick matrix $\left((1 - w_i \overline{w_j})k(x_i,x_j)\right)_{i,j=1}^n$ is positive. It is rather straightforward to prove that the Pick matrix is positive if such a contractive multiplier exists. The converse usually fails. We will say that the space has the complete Pick property if the claim holds for matrix-valued multipliers of all sizes.

We add the relatively common hypothesis that our complete Pick spaces are irreducible. Namely, no two kernels are linearly dependent, and no two are orthogonal (see \cite{AglerMcCarthy-book}). The theorem of Agler and McCarthy \cite[Theorem 4.2]{AglerMcCarthy-completeNP} states that if $\cH_k$ is a complete Pick space, then there exists $d \in \N \cup \{\infty\}$ and a map $b \colon X \to \B_d$, such that the kernel $k$ is equivalent to the pullback of the kernel of the Drury-Arveson space via $b$. Therefore, in particular, the space is unitarily equivalent (as a reproducing kernel Hilbert space) to a multiplier-coinvariant subspace of $H^2_d$.

This development led Davidson, Ramsey, and Shalit \cite{DRS1, DRS2} to study subvarieties of the $\B_d$ cut out by multipliers and the corresponding reproducing kernel Hilbert spaces and multiplier algebras. They have shown that given two subvarieties $V, W \subset \B_d$ with ideals $I_V, I_W \subset \cM_d$ of multipliers that vanish on $V$ and $W$, respectively, the quotients $\cM_d/I_V$ and $\cM_d/I_W$ are completely isometrically isomorphic if and only if there exists an automorphism $\varphi \in \Aut(\B_d)$, such that $\varphi(V) = W$. It follows from \cite[Remark 9.3]{DRS1} that every $\varphi \in \Aut(\B_d)$ gives rise to a unitary $U_{\varphi} \in B(H^2_d)$ whose adjoint is a linear extension of 
$$ U_{\varphi}^* k_{\lambda} = \sqrt{1 - \|\varphi^{-1}(0)\|^2} k_{\lambda}(\varphi^{-1}(0)) k_{\varphi(\lambda)} . $$ In particular, if $\varphi = \varphi_{\mu}$ for some $\mu \in \B_d$, then $U_{\varphi}^* 1 = \sqrt{1 - \|\mu\|^2} k_{\mu}$ which is the normalized kernel at $\mu$. Moreover, for every $f \in \cM_d$, if $M_f$ is the multiplication operator by $f$, then $U_{\varphi} M_f U_{\varphi}^* = M_{f \circ \varphi}$. As mentioned in the introduction, this result was also obtained by Rochberg \cite{Rochberg-hyp_geom} in the case of finite-dimensional spaces.

\subsection{Finite-dimensional complete Pick spaces} \label{subsec:comp_pick}

Following Ofek, Pandey, and Shalit \cite{OPS}, we will be interested in finite-dimensional complete Pick spaces. Recall that if $\cH_k$ is a complete Pick space on $X$ with $|X| = n$, then by the Agler-McCarthy theorem, there exists an embedding $b \colon X \to \B_d$ (including the possibility that $d = \infty$) of $ X $ into $ \B_d $. We may also choose an equivalent kernel and assume that $b(x_1) = 0$. The points of $b(X)$ then span at most an $n-1$-dimensional subspace, implying that we can choose $d \leq n-1$ (see also \cite[Theorem 3.2]{Rochberg-tetrahedra} for a more precise and detailed result). Therefore, we will work only with $d \in \N$ from now on.

Let $X_1$ and $X_2$ be sets and $\cH_{k_1}$ and $\cH_{k_2}$ be RKHSs on them. Let $\varphi \colon X_1 \to X_2$ be a map. Following \cite{OPS}, we define a morphism of RKHSs to be a bounded linear operator of the form $T k_{1,x} = f(x) k_{2,\varphi(x)}$, where $f \colon X_1 \to \C$ is a bounded function. An isomorphism of RKHSs is one that arises from an invertible $\varphi$ and non-vanishing $f(x)$. In particular, an isomorphism is isometric if and only if $k_2$ is a rescaling of $k_1$ in the sense of \cite[Section 2.6]{AglerMcCarthy-book}. Based on this definition, Ofek, Pandey, and Shalit have defined a version of the Banach-Mazur distance for a complete Pick space of dimension $n$:
\[
\rho_{RK}\left(\cH_{k_1}, \cH_{k_2} \right) = \log \left(\inf \left\{ \|T\| \|T^{-1}\| \mid T \colon \cH_{k_1} \to \cH_{k_2} \text{ is an RKHS isomorphism} \right\}\right).
\]
Let us denote the space of all complete Pick RKHS of dimension $n$ by $\pick_n$. Let $\B_d^{n,0}$ be the subset of the product consisting of distinct points. Namely,
\[
\B_d^{n,0} = \left\{ (z_1,\ldots,z_n) \in \B_d^n \mid \forall\, i \neq j,\, z_i \neq z_j \right\}.
\]
This set is the complement of the ramification locus of the symmetrization map. We denote by $\bX_{d,n}^0$ the image of $\B_d^{n,0}$ in $\bX_{d,n}$. Since every complete Pick RKHS on $n$ points embeds in $\B_{n-1}$, we can define a surjective map $\Phi_n \colon \bX_{n-1,n}^0 \to \pick_n$ that sends a set of $n$ distinct points to the corresponding complete Pick space. To be more precise, if $[X]= \sum_{j=1}^n [\lambda_j] \in \bX_{n-1,n}^0$, then $\Phi_n([X]) = [\Span\{k_{\lambda_1},\ldots,k_{\lambda_n}\}]$. Here, the brackets on the right-hand side indicate the equivalence class. Note that the fiber of $\Phi_n$ over a point $\cH$ consists of all subspaces of $H^2_d$ spanned by $n$ kernel functions that are isomorphic as RKHSs to $\cH$. In particular, the fiber is a homogeneous space under $\Aut(\B_{n-1})$ by the results of \cite{DRS2} and \cite{Rochberg-hyp_geom}. We note that $\Aut(\B_d)$ acts on $\B_d^n$ diagonally, and this action commutes with the actions of $S_n$. Hence, we get an action of $\Aut(\B_d)$ on $\bX_{d,n}$. Moreover, the set $\bX_{d,n}^0$ is $\Aut(\B_d)$-invariant. The map $\Phi_n$ described above is a quotient map (not in the topological sense, a priori). In \cite{OPS}, the authors define the invariant symmetric distance by
\[
\rho_s([X],[Y]) = \inf_{\varphi \in \Aut(\B_d)}d_s([X],\varphi([Y])).
\]
The action of $\Aut(\B_d)$ on $\B_d$ is proper by \cite[Example in Section 1.2]{Koszul-book}. Consequently, by \cite[Remark 1.4.1]{Koszul-book} the diagonal action of $\Aut(\B_d)$ on $\B_d^n$ is proper if and only if, for every $z \in \B_d^n$, there exists a neighborhood $U$ of $z$, such that the set $G(U|U) = \{\varphi \in \Aut(\B_d) \mid \varphi(U) \cap U \neq \emptyset\}$ is relatively compact. Since the action of $\Aut(\B_d)$ on $\B_d$ is proper, for $ (\lambda_1,\ldots,\lambda_n) \in \B_d^n $, we can choose neighborhoods $U_i \subset \B_d$ of $\lambda_i$, such that $G(U_i|U_i)$ is relatively compact. Thus setting $U = \prod_{i=1}^n U_i$, we get that $G(U|U) = \bigcap_{i=1}^n G(U_i|U_i)$ which is relatively compact. It is now easy to see that the action of $\Aut(\B_d)$ on $\bX_{d,n}$ and $\bX_{d,n}^0$ is proper \cite[Remark 1.1.1]{Koszul-book}. By \cite[Section 2]{Koszul-book}, every orbit of $\Aut(\B_d)$ in $\bX_{d,n}$ and $\bX_{d,n}^0$ is closed. Hence, the symmetric invariant metric induces a metric on the quotient $\bX_{d,n}^0/\Aut(\B_d)$. In fact, the metric induces the quotient topology. By \cite[Theorem 5.4]{OPS}, the topology induced by $\rho_{RK}$ on $\pick_n$ is the same as the quotient topology on $\bX_{n-1,n}^0/\Aut(\B_{n-1})$. However, the metrics are not equivalent by \cite[Example 6.3]{OPS}.

\subsection{Multivariable Blaschke multipliers} \label{subsec:Blaschke}

In \cite{McCulTrent}, McCullough and Trent show that every multiplier invariant subspace of the Drury-Arveson space is an image of a partially isometric multiplier $\Theta \colon H^2_d \otimes \cK \to H^2_d$, where $\cK$ is an auxiliary Hilbert space. We will need to use such multipliers. However, we will require a precise formula. Therefore, we use the idea of Alpay and Kaptano\u{g}lu. In \cite{AlpKapt}, Alpay and Kaptano\u{g}lu solve an interpolation problem on the unit ball of $\C^d$ and provide a concrete formula for the interpolating function. We will use their construction to show that the corresponding function is a partially isometric multiplier of McCullough and Trent.

To prove this, we will require some noncommutative techniques. The full Fock space $\bH^2_d$ is the completion of the free algebra $\C \langle z_1,\ldots,z_d\rangle$ with respect to the inner product that makes the monomials an orthonormal basis. The full Fock space is the noncommutative analog of the Hardy space (see, for example \cite{AriasPopescu-factorization,DavidsonPitts-inv,DavidsonPitts-NPint,DavidsonPitts-alg,Popescu-funcI,Popescu-funcII,JMS-bso}). In particular, the full Fock space is a noncommutative reproducing kernel Hilbert space in the sense of Ball, Marx, and Vinnikov \cite{BMV-ncrkhs} (see also \cite{SSS1}). The operators of left multiplication by the variables are isometries on $\bH^2_d$ that we will denote by $L_j$, $1 \leq j \leq d$. The left shift $$L = \begin{pmatrix} L_1 &\ldots & L_d \end{pmatrix} \colon \C^d \otimes \bH^2_d \to \bH^2_d$$ is a row isometry. In fact, by the results of Frazho \cite{Frazho}, Bunce \cite{Bunce}, and Popescu \cite{Popescu-dilations}, it is the universal pure row isometry. Similarly, the operators $R_j$, the operators of multiplication by the coordinates on the right, form a row isometry $R$, which we call the right shift. The WOT-closed algebra $\bH^{\infty}_d$ generated by the $L_j$ is the algebra of left multipliers on $\bH^2_d$. One can view both $\bH^2_d$ and $\bH^{\infty}_d$ as noncommutative functions on the row ball (quantization of the $\B_d$) (see \cite{KVV} for details on the theory of noncommutative functions and \cite{Popescu-funcI,Popescu-funcII,SSS1} for function theory on the row ball). In particular, Popescu in \cite{Popescu-characteristic} and independently Davidson and Pitts \cite{DavidsonPitts-inv} have proved a noncommutative version of the Beurling theorem. Namely, if $\cH \subset \bH^2_d$ is a right shift invariant subspace, then there exists a row isometry $\Theta \colon \bH^2_d \otimes \C^k \to \bH^2_d$, where $k \in \N \cup \{\infty\}$, such that the coordinates of $\Theta$ belong to $\bH^{\infty}_d$ and $\cH = \Theta \left(\bH^2_d \otimes \C^k\right)$.

As was already observed by Arveson \cite{Arv-subalg3}, the Drury-Arveson space embeds isometrically into the full Fock space as the symmetric tensors. We will identify $H^2_d$ with its image in $\bH^2_d$. It is quite easy to see that $H^2_d \subset \bH^2_d$ is both left and right coinvariant and that for all $1 \leq j \leq d$, $L_j^*|_{H^2_d} = M_{z_j}^* = R_j^*|_{H^2_d}$. Hence, if $X \subset \B_d$ and $\cH_X \subset H^2_d$ is the subspace spanned by $\{k_{\lambda} \mid \lambda \in X\}$, then $\bH^2_d \ominus \cH_X$ is both left and right invariant (see \cite{DavidsonPitts-alg,DavidsonPitts-NPint} for more details). Using the ideas of \cite{AlpKapt}, we will construct the isometric left multiplier with image $\bH^2_d \ominus \cH_X$. To do this, we require the fact proved by Voiculescu in \cite{Voiculescu-symmetries} (see also \cite{DavidsonPitts-alg,McCarTimoney,Popescu-aut,SSS1} for different points of view on this), that every automorphism $\varphi \in \Aut(\B_d)$ gives rise to a unitary $U_{\varphi} \in B(\bH^2_d)$ that extends the unitary on the Drury-Arveson space defined above. We refer the reader to \cite[Section 4]{DavidsonPitts-alg} for all the necessary information on these unitaries. In particular, for every $\varphi \in \Aut(\B_d)$, we set $\varphi(L) = U_{\varphi} L (I_d \otimes U_{\varphi}^*)$, Clearly, $\varphi(L)$ is a row isometry. 

\begin{lem} \label{lem:varphi_a}
Let $\lambda \in \B_d$ and let $\varphi_{\lambda}$ be the involution that exchanges $\lambda$ and $0$. Then the range of $\varphi_{\lambda}(L)$ is $\bH^2_d \ominus \C k_{\lambda}$.
\end{lem}
\begin{proof}
Let $\cH = \varphi_{\lambda}(L) (\C^d \otimes \bH^2_d)$. If $f \in \cH^{\perp}$, then for $G \in \C^d \otimes \bH^2_d$,
\[
0 = \langle \varphi_{\lambda}(L) G, f \rangle = \langle L (U_{\varphi_{\lambda}} \otimes I_d)^* G, U_{\varphi_{\lambda}}^* f \rangle.
\]
Since the orthogonal complement of the range of $L$ is spanned by the vacuum vector $1$, it follows that $U_{\varphi_{ \lambda }}^*f = \alpha$, for some $\alpha \in \C$ implying that $f =  \frac{\alpha}{\|k_{\lambda}\|} k_{\lambda}$ as $ \alpha U_{\varphi_{\lambda}} 1 = \frac{\alpha}{\|k_{\lambda}\|} k_{\lambda}$.
\end{proof}

Given $X = \{\lambda_1,\ldots,\lambda_n\} \subset \B_d$, we now construct a row isometry inductively. We set $F_1(L) = \varphi_{\lambda_1}(L)$, and assume for $ 2 \leq m \leq n$, that $F_{m-1}$ has been defined. Then, we define
\[
F_m(L) = F_{m-1}(L) (U_m \otimes I_{\bH^2_d}) A_m(L), \text{ where } A_m(L) = \begin{pmatrix} \varphi_{\lambda_m} & 0 \\ 0 & I_{(m-1)(d-1)} \end{pmatrix}.
\]
Here, $U_m$ is a scalar unitary of size $(m-1)(d-1) + 1$ satisfying $U_m e_1 = \frac{1}{\|F_{m-1}(\lambda_m)\|} F_{m-1}(\lambda_m)^*$. Observe that $A_m(L) \in M_{(m-1)(d-1) +1, (m-1)(d-1)+d}(\bH^{\infty}_d)$ and $F_m(L)$ is a row of size $(m-1)(d-1)+d$. Finally, we set $F_X(L) = F_n(L)$.

\begin{lem} \label{lem:product_distinct}
Let $X = \{\lambda_1, \ldots,\lambda_n\} \subset \B_d$ are distinct, then $F_X(L)$ is a row isometry and the range of $F_X(L)$ is $\bH^2_d \ominus \cH_X$.
\end{lem}
\begin{proof}
It is rather immediate that the $F_X$ is a row isometry by induction on $n$. Since $F_1$ is a row isometry, we take products of isometries.

We prove the second part of the claim by induction on $n$. The case $n=1$ is the content of the preceding lemma. To prove for $n$, we note that for a product of two isometries $V_1 V_2$, we have $\ker (V_2^* V_1^*) = \ker V_1^* \oplus V_1 (\ker V_2^*)$. Thus $\ker F_m(L)^* = \ker F_{m-1}(L)^* \oplus F_{m-1}(L)(\ker A_m(L)^* (U_m^* \otimes I_{\bH^2_d}))$. By induction, $\ker F_{m-1}(L)^* = \Span\{k_{\lambda_1}, \ldots, k_{\lambda_{m-1}}\}$. Moreover, we note that by Lemma \ref{lem:varphi_a}, the kernel of $\ker A_m(L)^* = \C k_{\lambda_m} e_1$. We conclude that $\dim \ker F_m(L)^* = m$. Now we compute
\[
F_m(L)^* k_{\lambda_m} = A_m(L)^* (U_m^* \otimes I_{\bH^2_d}) F_{m-1}(L)^* k_{\lambda_m} = A_m(L)^* (U_m^* \otimes I_{\bH^2_d}) (F_{m-1}(\lambda_m)^* \otimes k_{\lambda_m}) = 0.
\]
This completes the proof.
\end{proof}

Let us write $\bH^2_d = (\bH^2_d \ominus H^2_d) \oplus (H^2_d \ominus \cH_X) \oplus \cH_X$. Then, the range projection of $F_X(L)$ decomposes as $F_X(L) F_X(L)^* = 2 I - P_{H^2_d} - P_{\cH_X}$. Let $M_{b_X} = P_{H^2_d} F_X(L)|_{\C^{(m-1)(d-1)+d} \otimes H^2_d}$. Then, since $H^2_d$ is left multiplier coinvariant and $P_{\cH_X} \leq P_{H^2_d}$, we get that
\[
M_{b_X} M_{b_X}^* = P_{H^2_d} F_X(L) F_X(L)^*|_{H^2_d} = I - P_{\cH_X}. 
\]
In particular, $M_{b_X}$ is a partially isometric multiplier with range $H^2_d \ominus \cH_X$. We caution the reader that the notation $b_X$ is somewhat misleading. In fact, the row multiplier $b_X$ depends on the choice of order on $X$ and the choice of unitaries. However, when we use this construction, we will make explicit choices.

\section{The Grassmannian of coinvariant spaces} \label{sec:grassmannian}

In this section, we define the space that we will study along with its connection to $\bX_{d,n}$ and explore its basic properties.

\begin{dfn}
We define $\bP_{d,n} \subset B(H^2_d)$ as the space of all projections onto multiplier coinvariant $n$-dimensional subspaces. Namely,
\[
\bP_{d,n} = \left\{ P \in B(H^2_d) \mid P = P^2 = P^*,\, \forall f \in \cM_d:\, P M_f^* P = M_f^* P,\, \dim( P H^2_d) = n \right\}.
\]
The topology on $\bP_{d,n}$ is the norm topology.
\end{dfn}

Now, we study the general form of an $n$-dimensional multiplier coinvariant subspace of $H^2_d$. We are convinced that these results are well-known to experts. However, since we have not found a reference, we include them here for completeness.

\begin{lem} \label{lem:joint_kernel_n}
Let $ M_z = ( M_{ z_1, \hdots, M_{ z_d } } ) $ be the $ d $-tuple of multiplication operators by coordinate functions defined on $ H_d^2 $ and $ \lambda \in \B_d $. Then
    \[
    \bigcap_{|\alpha|=n}\ker(M_z^*- \bar{\lambda} I)^{\alpha}= \Span\left\{k^{(\beta)}_{\lambda} \mid |\beta|\leq n-1 \right\}.
    \]
\end{lem}
\begin{proof}
It is immediate to check that for every multi-index $\beta$ with $|\beta| < |\alpha|$ and every $f \in H^2_d$,
\[
\langle f, (M_z^* - \bar{\lambda} I)^{\alpha} k^{(\beta)}_{\lambda} \rangle = \dfrac{\partial^{\beta} }{\partial z^{\beta}} \left( (z - \lambda)^{\alpha} f \right)(\lambda) = 0.
\]
Hence, $\Span\left\{k^{(\beta)}_{\lambda} \mid |\beta|\leq n-1 \right\} \subset  \bigcap_{|\alpha|=n}\ker(M_z^*- \bar{\lambda} I)^{\alpha}$. Now, by considering appropriate polynomials, we can easily see that the set $\left\{k^{(\beta)}_{\lambda} \mid |\beta|\leq n-1 \right\}$ is linearly independent. Now let $f \in \bigcap_{|\alpha|=n}\ker(M_z^*- \bar{\lambda} I)^{\alpha}$. For every polynomial $p \in \C[z_1,\ldots,z_d]$, we can write $p = q + \sum_{|\alpha|=n} (z-\lambda)^{\alpha} q_{\alpha}$, where $\deg q  \leq n-1$. Hence, $\langle f, p \rangle = \langle f, q \rangle$. This implies that every functional on the joint kernel is a linear combination of inner products against $\{z^{\beta}\}_{|\beta| \leq n-1}$. Consequently, $\dim \bigcap_{|\alpha|=n}\ker(M_z^*- \bar{\lambda} I)^{\alpha} \leq \dim \Span\left\{k^{(\beta)}_{\lambda} \mid |\beta|\leq n-1 \right\}$ and we are done.
\end{proof}

\begin{prop} \label{prop:char_fin_dim}
Let $\cH \subset H^2_d$ be a multiplier coinvariant subspace with $\dim \cH = n$. Then there exist points $\lambda_1, \ldots, \lambda_m \in \B_d$, such that $\cH = \sum_{j=1}^m \cH_j$ ( a non-orthogonal direct sum), where for $1 \leq j \leq m$, $k_{\overline{\lambda_j}} \in \cH_j \subseteq \bigcap_{|\alpha|=n}\ker(M_z^*- \lambda_j I)^{\alpha}$
\end{prop}
\begin{proof}
Let $\cH \subset H^2_d$ be a finite-dimensional multiplier coinvariant subspace. Let us consider the tuple $A = M_z^*|_{\cH}$. This tuple is a commuting pure column contraction, and its spectrum is contained in $\B_d$. Let $\cH = \sum_{j=1}^k R_{\lambda_j}(A)$ be the decomposition of $\cH$ into root subspaces of $A$. Then, by the previous lemma $k_{\overline{\lambda_j}} \in R_{\lambda_j}(A)$. Moreover, $R_{\lambda_j}(A) \subseteq \bigcap_{|\alpha| = \dim \cH} \ker(M_z^* - \lambda_j)^{\alpha}$. Therefore, $R_{\lambda_j}(A)$ is spanned by a kernel function and some directional derivatives of it at $\overline{\lambda_j}$.
\end{proof}

\begin{remark}
The root subspaces of a single matrix are spanned by Jordan chains and eigenvectors. In our case, every joint eigenvalue has a one-dimensional eigenspace. Hence, the directional derivatives that span $R_{\lambda_j}(A)$ are a form of generalized Jordan chain.
\end{remark}

Now note that the above discussion leads us to define a map from $\bP_{d,n}$ to $\bX_{d,n}$.

\begin{dfn}
We define a map $\Psi_{d,n} \colon \bP_{d,n} \to \bX_{d,n}$ given by $\Psi_{d,n}(P) = \sigma_J(P M_z|_{P H^2_d})$.
\end{dfn}

\begin{thm} \label{thm:proj_continuous}
The map $\Psi_{d,n}$ is surjective and continuous.
\end{thm}
\begin{proof}
Let us prove continuity first. Let $\{P_k\}_{k=1}^{\infty} \subset \bP_{d,n}$ be a sequence that converges to some $P \in \bP_{d,n}$. Let us fix an isometry $W \colon \C^n \to H^2_d$ with $\mathrm{ran} W = P H^2_d$. Applying Lemma \ref{lem:dist_isom}, we obtain a sequence of isometries $V_k \colon \C^n \to H^2_d$ with $\mathrm{ran} V_k = P_k H^2_d$ and, moreover,
\[
\|P_k - P\| \leq \|V_k - W\| \leq \sqrt{2} \|P_k - P\|.
\]
Set $A_k = V_k^* M_z (I_d \otimes V_k)$ and $A = W^* M_z (I_d \otimes W)$. We note that $\sigma_J(A_k) = \Psi_{d,n}(P_k)$ and $\sigma_J(A) = \Psi_{d,n}(P)$. It is immediate that $\|A_k - A\|_{row} \leq 2 \sqrt{2} \|P_k - P\|$. By  Proposition \ref{prop:convergence_of_spectra}, we get that $\lim_{n \to \infty} d_o(\sigma_J(A_k), \sigma_J(A)) = 0$. Since $A$ is a strict row contraction, its joint spectrum is contained in some $r \overline{\B_d}$ for $0 < r < 1$. Changing $r$ slightly, we see that for $k$ big enough, the joint spectrum of the $A_k$ is contained in $r \overline{\B_d}$. We conclude by applying Lemma \ref{lem:opt_vs_sym} that $\Psi_{d,n}(P_k) \to \Psi_{d,n}(P)$.

To prove surjectivity, we fix an arbitrary $\lambda \in \B_d$ and $n \in \N$, and consider the subspace $\cH_{n, \lambda}^{(1)} = \Span\{k_{\lambda}, k_{\lambda}^{(e_1)},\ldots,k_{\lambda}^{(n e_1)}\}$, where $e_1$ is the first vector of the standard basis. We argue that $\cH_{n,\lambda}^{(1)}$ is multiplier coinvariant. If this is true, then $\Psi_{d,n}\left(P_{\cH^{(1)}_{n,\lambda}}\right) = n [\lambda]$. Now, sums of such spaces can give us any point with multiplicity. It is immediate that points with no multiplicity are in the image of $\Psi_{d,n}$. Finally, for $f \in \cM_d$ and $g \in H^2_d$, we have that
\[
\langle g, M_f^* k_{\lambda}^{(n e_1)} \rangle = \dfrac{\partial^n (fg)}{\partial z_1^n}(\lambda) = \sum_{j=0}^n \binom{n}{j} \dfrac{\partial^{n-j} f}{\partial z_1^{n-j}} (\lambda) \dfrac{\partial^j g}{\partial z_1^j}(\lambda) = \sum_{j=0}^n \binom{n}{j} \dfrac{\partial^{n-j} f}{\partial z_1^{n-j}} (\lambda) \langle g, k_{\lambda}^{(j e_1)} \rangle.
\]
verifying that $M_f^* k_{\lambda}^{(n e_1)} \in \cH_{n,\lambda}^{(1)}$. 
\end{proof}

\begin{remark}
It is clear that the choice of $e_1$ in the proof above is arbitrary.
\end{remark}

\begin{example}
Let $\cH_n = \Span\{1, k_{1/n}\} \subset H^2(\D)$ and $\cH = \Span\{1, z\} \subset H^2(\D)$. Let $P_n = P_{\cH_n}$ and $P = P_{\cH}$. It is clear that $\Psi_{1,2}(P_n) = [0] + [1/n]$ and $\Psi_{1,2}(P) = 2[0]$. To see what happens on the side of the projections, we set
\[
f_n(z) = \dfrac{k_{1/n} - 1}{\|k_{1/n} - 1 \|} = \sqrt{n^2-1}(k_{1/n} - 1) = \dfrac{\sqrt{n^2-1}}{n}\dfrac{z}{1 - z/n}.
\]
It is clear now that $\{1, f_n\}$ is an orthonormal basis for $\cH_n$. Now, we note that
\[
\|f_n - z\|^2 = \left\|\dfrac{\sqrt{n^2-1}}{n}\dfrac{1}{1 - z/n} - 1 \right\|^2 = 2 \left( 1 - \dfrac{\sqrt{n^2-1}}{n}\right).
\]
In particular, $\lim_{n\to\infty} f_n = z$. Now for a unit vector $f \in H^2(\D)$,
\[
\|P_n f - P f\| = \left\|\langle f, f_n \rangle f_n - \langle f, z \rangle z\right\| \leq |\langle f, f_n \rangle| \|f_n - z\| + |\langle f, f_n - z \rangle| \rightarrow 0.
\]
This shows, in particular, that the distance induced by the norm on finite-dimensional complete Pick RKHSs is different from the Banach-Mazur distance of Ofek, Pandey, and Shalit by \cite[Example 6.2]{OPS}
\end{example}

The following example shows that for $d \geq 2$, the map $\Psi_{d,n}$ is not invertible.

\begin{example}
Consider the following subspaces of $H^2_2$: $\cH_{1,n} = \Span\{1, k_{1/n e_1}\}$, $\cH_{2,n} = \Span\{1, k_{1/n e_2} \}$, $\cH_1 = \Span\{1,z_1\}$, and $\cH_2 = \Span\{1,z_2\}$. Denote by $P_{1,n}$, $P_{2,n}$, $P_1$, and $P_2$ the respective projection. As in the previous example, it is not hard to show that $P_{1,n} \rightarrow P_1$ and $P_{2,n} \rightarrow P_2$. Clearly, $\Psi_{2,2}(P_1) = \Psi_{2,2}(P_2) = 2 [0]$, $\Psi_{2,2}(P_{1,n}) = [0] + [1/n e_1]$, and $\Psi_{2,2}(P_{2,n}) = [0] + [1/n e_2]$. The later sequences converge to $2[0]$.
\end{example}

\begin{remark}
Let $d \geq 2$. The fiber $\Psi_{d,2}^{-1}(2 [0])$ can easily be described using Proposition \ref{prop:char_fin_dim} as the collection of projection onto spaces of the form $\Span\{1, f\}$, where $f$ is a homogeneous polynomial of degree $1$. It is easy to see that the map $\Psi_{d,2}^{-1}(2 [0]) \to \bP^{d-1}(\C)$, that sends a projection to the point corresponding to the coefficients of $f$ is a homeomorphism. In this sense, the map $\Psi_{d,2}$ and, more generally, $\Psi_{d,n}$ are akin to blowdown maps from algebraic geometry.
\end{remark}

\section{The subspace $\bP_{d,n}^0$} \label{sec:configuration_space}

As we saw, a naive attempt to construct an inverse map will immediately run into an obstacle for $d \geq 2$. The intuitive reason is that though the space $\bX_{d,n}$ keeps track of multiplicities of points in $\B_d$, the space $\bP_{d,n}$ also keeps track of ``directions'' at each point with multiplicity. Since multiple points are problematic, we must restrict our attention to $\bX_{d,n}^0$.

\begin{dfn}
We set $\bP_{d,n}^0 = \Psi_{d,n}^{-1}\left(\bX_{d,n}^0\right)$. In particular, This is the space of projections onto multiplier coinvariant $n$-dimensional subspaces of $H^2_d$ spanned by $n$ distinct kernel functions.
\end{dfn}

Given $[X] \in \bX_{d,n}^0$, let us fix $b_X \colon \B_d \to \B_{(n-1)(d-1)+d}$ a multiplier, such that $M_{b_X}$ is a partial isometry with range $H^2_d \ominus \cH_X$ obtained in Subsection \ref{subsec:Blaschke}. Recall that at each step of the construction of $b_X$, there was a choice of a unitary. We will show that this choice does not affect the corresponding projection. 

\begin{lem} \label{lem:not_far_unitary}
Let $\cH$ be a Hilbert space and fix a unit vector $w \in \cH$. For every $\varepsilon > 0$, if $v_1, v_2 \in \cH$ are unit vectors with $\|v_1 - v_2\| < \varepsilon$, then for every unitary $U_1 \in B(\cH)$, such that $U_1 w = v_1$, there exists a unitary $U_2 \in B(\cH)$, such that $U_2 w = v_2$ and $\|U_1 - U_2\| < \sqrt{2} \varepsilon$.
\end{lem}
\begin{proof}
If $v_2 = e^{i \theta} v_1$, then we set $U_2 =e^{- i \theta} U_1$. Hence, $\|U_1 - U_2\| = |1 - e^{-i \theta}| = \|v_1 - v_2\|$. Therefore, we will assume from now on that $v_1$ and $v_2$ are linearly independent.

Since we can consider $U_1^*$ and $U_2^*$, we will assume that $U_1 v_1 = w$. Now let $\cK = \{v_1,v_2\}^{\perp}$. We set $U_2|_{\cK} = U_1|_{\cK}$.  Let $v_1' =\frac{v_2 - \langle v_2,v_1 \rangle v_1}{\|v_2 - \langle v_2,v_1 \rangle v_1\|}$ and $v_2' = - \frac{v_1 - \langle v_1,v_2 \rangle v_2}{\|v_1 - \langle v_1,v_2 \rangle v_2\|}$. Let $U_1 v_1' = w'$ and set $U_2 v_2 = w$ and $U_2 v_2' = w'$. Note that
\[
\|U_1 v_2 - U_2 v_2\| = \|U_1 v_2 - U_1 v_1 + U_1 v_1 - w\| \leq \|v_2 - v_1\|.
\]
Similarly, $\|U_1 v_2' - U_2 v_2'\| \leq \|v_1' - v_2'\| = \|v_2 - v_1\|$. 
\begin{comment}
To see the last equality note that
\[
\| v_2 - \langle v_2, v_1 \rangle v_1\|^2 = \langle v_2 - \langle v_2, v_1 \rangle v_1, v_2 - \langle v_2, v_1 \rangle v_1 \rangle = 1 - |\langle v_2, v_1 \rangle|^2.
\]
Similarly, $\| v_1- \langle v_1, v_2 \rangle v_2\| = \sqrt{1 - |\langle v_2, v_1 \rangle|^2}$ Moreover,
\[
\langle v_2, v_1'\rangle = \dfrac{1}{\sqrt{1 - |\langle v_2, v_1 \rangle|^2}} \langle v_2,  v_2 - \langle v_2, v_1 \rangle v_1 = \sqrt{1 - |\langle v_2, v_1 \rangle|^2}
\]
Hence,
\[
\|v_1' - v_2'\|^2 = \|\langle v_1',v_2 \rangle v_2 + \langle v_1', v_2' \rangle v_2' - v_2'\|^2 = 1 - |\langle v_2, v_1 \rangle|^2 + |1 - \langle v_1', v_2' \rangle|^2.
\]
Now,
\begin{multline*}
\langle v_1',v_2' \rangle = -\dfrac{1}{1 - |\langle v_2,v_1 \rangle|^2} \langle v_2 - \langle v_2, v_1 \rangle v_1, v_1 - \langle v_1, v_2 \rangle v_2 \rangle = \\ -\dfrac{1}{1 - |\langle v_2,v_1 \rangle|^2} \left( \langle v_2, v_1 \rangle - \langle v_2, v_1 \rangle - \langle v_2, v_1 \rangle + \langle v_2, v_1 \rangle |\langle v_2, v_1 \rangle|^2 \right) = \langle v_2, v_1 \rangle.
\end{multline*}
Hence,
\[
\|v_1' - v_2'\|^2 = 1 - |\langle v_2, v_1 \rangle|^2 + 1 - 2 \Re \langle v_2, v_1 \rangle + |\\langle v_2, v_1 \rangle|^2 = 2 - 2 \Re \langle v_2, v_1 \rangle = \|v_2 - v_1\|^2
\]
\end{comment}
Therefore, writing every unit vector $x = \alpha v_2 + \beta v_2' + v$ with $v \in \cK$, we have that
\[
\|U_1 x - U_2 x\| \leq |\alpha| \|U_1 v_2 - U_2 v_2\| + |\beta| \|U_1 v_2 ' - U_2 v_2'\| \leq \sqrt{2}\|v_1 - v_2\|.
\]
\end{proof}

\begin{prop} \label{prop:Blaschke_choice}
Let $[X_m] \in \bX_{d,n}^0$ be a sequence that converges to $[X] \in \bX_{d,n}^0$. Fix $b_X$, as above. Then, there exists a choice of $b_{X_m}$, such that $b_{X_m} \to b_X$ uniformly on compacta in $\B_d$.
\end{prop}
\begin{proof}
Let us fix a preimage $X = \{\lambda_1,\ldots,\lambda_n\}$ of $[X]$. By \cite[Proposition 3.3]{OPS}, if we set $\delta = \min\{d_{ph}(\lambda_i, \lambda_j) \mid 1 \leq i,j \leq n\}$, then for $d_s([X],[X_m]) < \delta/2$, the symmetric distance coincides with the Hausdorff distance. In particular, in every pseudohyperbolic ball of radius less than $ \delta/2$ around $\lambda_i$, there exists precisely one point of $[X_m]$ for $m$ big enough. In other words, up to ignoring a prefix of the sequence, we may choose a preimage $X_m = \{\lambda_{1,m}, \ldots, \lambda_{n,m}\}$, such that for all $1 \leq j \leq n$, $\lim_{m \to \infty} \lambda_{j,m} = \lambda_j$.

Now we observe that it is not hard to show that if a sequence of points $\mu_m \in \B_d$ converges to $\mu$, the corresponding involutions $\varphi_{\mu_m}$ converge uniformly on compacta to $\varphi_{\mu}$.

Now we prove the claim by induction on $n$. The base case is handled by the above paragraph. Let us write $X' = \{\lambda_1,\ldots,\lambda_{n-1}\}$ and, similarly, $X_m = \{\lambda_{1,m} , \ldots, \lambda_{n-1,m}\}$. Let us assume that we have constructed $b_{X_m'}$, such that $b_{X_m'}$ converges uniformly on compacta in $\B_d$ to $b_{X'}$. Now take $0 < r < 1$ and fix an arbitrary $\varepsilon > 0$. Since $\lim_{m\to\infty} \lambda_{n,m} = \lambda_n$, for big enough $m$, we have 
$$\|b_{X_m'}(\lambda_{n,m}) - b_{X'}(\lambda_n)\| < \frac{\varepsilon}{3 \sqrt{2}} . $$ 
In particular, the vectors $v_m = \frac{1}{\|b_{X_m'}(\lambda_{n,m})\|}b_{X_m'}(\lambda_{n,m})$ and $v = \frac{1}{\|b_{X'}(\lambda_{n})\|}b_{X'}(\lambda_{n})$ are close. By construction,
\[
b_X(z) = b_{X'}(z) U \begin{pmatrix} \varphi_{\lambda_n}(z) & 0 \\ 0 & I_{(n-1)(d-1)} \end{pmatrix}.
\]
Here, $U$ is a unitary satisfying $U e_1 = v$. Let $U_m$ be a unitary obtained from Lemma \ref{lem:not_far_unitary} such that $U_m e_1 = v_m$ and $\|U - U_m\| < \varepsilon/3$. Set
\[
b_{X_m}(z)  = b_{X_m'}(z) U_m \begin{pmatrix} \varphi_{\lambda_{n,m}(z)} & 0 \\ 0 & I_{(n-1)(d-1)}\end{pmatrix}.
\]
Then,
\begin{eqnarray*}
\|b_X(z) - b_{X_m}(z)\| & = & \left\| b_{X'}(z) U\begin{pmatrix} \varphi_{\lambda_n}(z) & 0 \\ 0 & I_{(n-1)(d-1)} \end{pmatrix} - b_{X_m'}(z) U_m \begin{pmatrix} \varphi_{\lambda_{n,m}(z)} & 0 \\ 0 & I_{(n-1)(d-1)}\end{pmatrix}\right\| \\ 
& \leq & \left\| U\begin{pmatrix} \varphi_{\lambda_n}(z) & 0 \\ 0 & I_{(n-1)(d-1)} \end{pmatrix} - U_m \begin{pmatrix} \varphi_{\lambda_{n,m}(z)} & 0 \\ 0 & I_{(n-1)(d-1)}\end{pmatrix}\right\| \\
& + & \|b_{X'}(z) - b_{X_m'}(z)\| \\
& \leq & \|\varphi_{\lambda_n}(z) - \varphi_{\lambda_{n,m}(z)}\| + \|b_{X'}(z) - b_{X_m'}(z)\| + \dfrac{\varepsilon}{3}.
\end{eqnarray*}
Now, by induction, we have uniform convergence for $|z| \leq r$ of $b_{X_m'}$ to $b_{X'}$. Hence, for $m$ big enough, the above expression is less than $\varepsilon$ for all $|z| \leq r$.
\end{proof}

We note that though $b_X$ is not uniquely defined by $X$, the operator $I - M_{b_X} M_{b_X}^*$ is the projection onto the space spanned by the kernels at the points of $X$. Therefore, the map in the following theorem is well-defined.

\begin{thm} \label{thm:inverse_map}
Define a map $\Phi_{d,n} \colon \bX_{d,n}^0 \to \bP_{d,n}^0$ by $[X] \mapsto I - M_{b_X} M_{b_X}^*$. Then $\Phi_{d,n}$ is continuous, $\Psi_{d,n} \circ \Phi_{d,n} = \mathrm{Id}_{\bX_{d,n}^0}$, and $\Phi_{d,n} \circ \Psi_{d,n}|_{\bP_{d,n}^0} =  \mathrm{Id}_{\bP_{d,n}^0}$.
\end{thm}
\begin{proof}
Let $ \{ [X_m] \} \in \bX_{d,n}^0$ be a sequence that converges to $[X] \in \bX_{d,n}^0$. By \cite[Lemma 5]{Dixmier_Douady}, the norm and SOT topologies restricted to $\bP_{d,n}$ coincide. Hence, it suffices to show that $M_{b_{X_m}} M_{b_{X_m}}^* \stackrel{SOT}{\rightarrow} M_{b_X} M_{b_X}^*$. Fix some preimage $X \in \B_d^n$ of $[X]$ and construct $b_X$. By the preceding proposition, we can choose preimages $X_m$ and construct $b_{X_m}$, such that $b_{X_m}$ converges to $b_X$ uniformly on compacta. Since the sequence is bounded, we have that $M_{b_{X_m}} \stackrel{WOT}{\longrightarrow} M_{b_X}$. It is then a well known fact on multipliers that $M_{b_{X_m}}^* \stackrel{SOT}{\longrightarrow} M_{b_X}^*$. Again since the sequences are bounded, we have that $M_{b_{X_m}} M_{b_{X_m}}^* \stackrel{SOT}{\longrightarrow} M_{b_X} M_{b_X}^*$ as desired.
\end{proof}

The case of the Hardy space $H^2(\D)$ is different. The intuitive reason is that there are no (complex) directions to obstruct the inverse map in the case of the disc. Given $X \in \bX_{1,n}$, we can define the associate Blaschke product counting multiplicities.

\begin{prop} \label{prop:hardy_space}
The map $\Psi_{1,n} \colon \bP_{1,n} \to \bX_{1,n}$ is a homeomorphism.
\end{prop}
\begin{proof}
The proof follows the same lines as the proof of the theorem. However, it is much less complicated, for let $[X]= \sum_{j=1}^m n_j[\lambda_j] \in \bX_{1,n}$, then we can associate to $[X]$ the Blaschke product 
\[
b_X(z) = \prod_{j=1}^m \left(\dfrac{z - \lambda_j}{1 - \overline{\lambda_j} z}\right)^{n_j}.
\]
This allows us to extend the argument in the theorem to the points with multiplicity, as well.
\end{proof}

\begin{remark}
The group $\Aut(\B_d)$ acts on $\bP_{d,n}$ by $\varphi \cdot P = U_{\varphi}^* P U_{\varphi}$. To see this, we first note that $\varphi \cdot P$ is a projection with $n$-dimensional range. Moreover, for every $f \in \cM_d$,
\[
(\varphi \cdot P) M_f^* (\varphi \cdot P) = U_{\varphi}^* P U_{\varphi} M_f^* U_{\varphi}^* P U_{\varphi} = U_{\varphi}^* P M_{f \circ \varphi}^* P U_{\varphi} = U_{\varphi}^* M_{f \circ \varphi}^* P U_{\varphi} = M_f^* (\varphi \cdot P).
\]
Now for $P \in \bP_{d,n}^0$, with $\Psi_{d,n}(P) = \sum_{j=1}^n [\lambda_j]$ and $\varphi = \varphi_{\mu}$, we have
\begin{eqnarray*}
(\varphi_{\mu} \cdot P) k_{\varphi_{\mu}(\lambda_j)} = (U_{\varphi_{\mu}} P U_{\varphi_{\mu}}) k_{\varphi_{\mu}(\lambda_j)} & = & \sqrt{1 - \|\mu\|^2} (U_{\varphi_{\mu}} P) \overline{k_{\mu}(\varphi_{\mu}(\lambda_j))} k_{\lambda_j} \\  
& = & \dfrac{\sqrt{1 - \|\mu\|^2}}{1 - \langle \mu, \varphi_{\mu}(\lambda_j) \rangle} U_{\varphi_{\mu}} k_{\lambda_j} \\
& = & \dfrac{1- \|\mu\|^2}{(1 - \langle \mu, \varphi_{\mu}(\lambda_j) \rangle)(1 - \langle \mu, \lambda_j \rangle)} k_{\varphi_{\mu}(\lambda_j)}.
\end{eqnarray*}
Note that by \cite[Theorem 2.2.2(iii)]{Rudin-book}, 
\[
1 - \langle \mu, \varphi_{\mu}(\lambda_j) \rangle = 1 - \langle \varphi_{\mu}(0), \varphi_{\mu}(\lambda_j) \rangle = \dfrac{1 - \|\mu\|^2}{1 - \langle \mu, \lambda_j \rangle}.
\]
Hence, $(\varphi_{\mu} \cdot P) k_{\varphi_{\mu}(\lambda_j)} = k_{\varphi_{\mu}(\lambda_j)}$. Therefore, $\Psi_{d,n}(\varphi_{\mu} \cdot P) = \sum_{j=1}^n [\varphi_{\mu}(\lambda_j)]$. The same claim holds if $\varphi$ is a unitary. Hence, we get that $\Psi_{d,n}$ is $\Aut(\B_d)$-equivariant. From \cite[Remark 1.1.1]{Koszul-book} we conclude that the action of $\Aut(\B_d)$ on $\bP_{d,n}$ is proper. Moreover, the results above combine to show that $\bP_{n-1,n}^0/\Aut(\B_d)$ is homeomorphic to $\pick_n$.
\end{remark}

\section{Bundles of complete Pick spaces and their multiplier algebras} \label{sec:bundles}

In this section, we will use classical results on Hilbert and Banach space bundles to discuss deformations of complete Pick spaces and the corresponding multiplier algebras. Some old but good references on bundles are \cite{Dixmier_Douady} and \cite{Dixmier-book}. 

The tautological bundle on $\bP_{d,n}$ is $P \mapsto PH^2_d$. This is a vector bundle by \cite[Proposition 2.3]{Dupre}. We will denote the tautological bundle by $\cE_{d,n}$. Note that $\cE_{d,n}$ is the restriction of the tautological bundle on the strong Grassmannian of $H^2_d$ in the sense of Dixmier and Douady. Namely, if $\Gr_f$ is the collection of all projections in $B(H^2_d)$ equipped with the strong topology, then $\bP_{d,n}$ embeds into $\Gr_f$ since the strong and norm topologies on $\bP_{d,n}$ coincide. Moreover, we can consider $\fM_{d,n} = \bigsqcup_{P \in \bP_{d,n}} P \cM_d|_{P H^2_d}$. Note that each $f \in \cM_d$ gives rise to a section of $\fM_{d,n}$ given by $\gamma_f(P) = P M_f|_{P H^2_d}$. We denote by $\Lambda$ the collection of all such sections. By \cite[Proposition 10.2.3]{Dixmier-book}, it suffices to show that the map $P \mapsto \|\gamma_f(P)\|$ is continuous to deduce that $\Lambda$ generates a Banach algebra bundle structure on $\fM_{d,n}$. The continuity of the norm is, however, immediate since the topology on $\bP_{d,n}$ is of norm convergence. By \cite[Theorem II.13.18]{FellDoran-v1}, there is a unique topology on $\fM_{d,n}$, such that the Banach bundle structure $\Gamma$ generated by $\Lambda$ consists of continuous sections. We denote by $\Gamma_b$ the collection of all bounded continuous sections and note that $\Lambda \subset \Gamma_b$. We equip $\Gamma_b$ with the norm $\|x\| = \sup_{P \in \bP_{d,n}} \|x(P)\|$. It is then clear that $\Gamma_b$ is a Banach algebra.

For every $m \in \N$, we can consider the construction above for $M_m(\cM_d)$. Namely, we consider the bundle $P \mapsto (I_m \otimes P) M_m(\cM_d)|_{\C^m \otimes P H^2_d}$. As mentioned above, the collection $M_m(\Lambda)$ generates a Banach algebra structure on this bundle. The collection of continuous sections can be identified with $M_m(\Gamma)$. Therefore, we can identify this bundle with $M_m(\fM_{d,n}) \cong M_m \otimes \fM_{d,n}$. We summarize the discussion in the following proposition.
\begin{prop} \label{prop:fd_op_alg_bundle}
We have a Hilbert space bundle $\cE_{d,n}$ and a Banach algebra bundle $\fM_{d,n}$ on $\bP_{d,n}$. Moreover, for every $m \in \N$, we have a Banach algebra bundle $M_m(\fM_{d,n})$. If $\Gamma_b$ is the collection of all bounded continuous sections of $\fM_{d,n}$ equipped with the sup norm, then $\Gamma_b$ is an operator algebra.
\end{prop}
\begin{proof}
We need only to prove that $\Gamma_b$ is an operator algebra. To see that $\Gamma_b$ is an operator space, we note that $M_m(\Gamma_b)$ is precisely the collection of bounded continuous sections of $M_m(\fM_{d,n})$ as obtained in the construction above. Consequently, every fiber turns out to be an operator algebra. Therefore, by Ruan's axioms and the condition of the Blecher-Ruan-Sinclair theorem \cite{Paulsen-book}, it follows that $\Gamma_b$ is an operator algebra.
\end{proof}

At this point, we would like to point the reader's attention to a natural question. We note that if $Z$ is a paracompact space and $\varphi \colon Z \to \bP_{d,n}$ is a continuous map, then we obtain a vector bundle on $Z$ via $\cF = \varphi^* \cE_{d,n}$. If we, further, assume that $\varphi(Z) \subset \bP_{d,n}^0$, the bundle $\varphi^* \cE_{d,n}$ is a bundle of complete Pick spaces on $Z$. If we identify $\bP_{d,n}^0$ with $\bX_{d,n}^0$ via the maps from Theorem \ref{thm:inverse_map}, then we have a map $\varphi \colon Z \to \bX_{d,n}^0$. Recall that $\B_d^{n,0}$ is the collection of all $n$-tuples of distinct points in $\B_d$. If we write $\pi_{d,n}$ for the symmetrization map, then $\pi_{d,n} \colon \B_d^{n,0} \to \bX_{d,n}^0$ is a principal $S_n$ bundle. Therefore, we can obtain a principal $S_n$-bundle on $Z$ via $\B_d^{n,0} \times_{\bX_{d,n}^0} Z$ and the projection map onto the second coordinate, where the fiber product is constructed, of course, using $\varphi$. Let $z_0 \in Z$ and $\varphi(z_0) \in V \subset \bX_{d,n}^0$ be a neighborhood, such that the principal bundle is trivialized over it. We can construct a frame of $\cF$ on $\phi^{-1} (V)$ by taking the sections $z \mapsto k_{\lambda_j(z)}$, where $\lambda_1(z),\ldots,\lambda_n(z)$ is the trivialization of the $S_n$-bundle on $\phi^{-1} (V)$. Choosing a trivializing cover $Z = \bigcup_{\alpha} U_{\alpha}$ for the $S_n$-bundle, we can see that the structure group of $\cF$ can be reduced to $S_n$. Moreover, for every $\alpha$, we obtain the matrix-valued function $k_{\alpha}(z) = \left( \langle k_{\lambda_j(z)}, k_{\lambda_i(z)} \right)_{i,j=1}^n$. For every $\alpha$ and $\beta$ in our index set, there exists a permutation matrix $h_{\alpha,\beta}$, such that $h_{\alpha,\beta}^* k_{\alpha} h_{\alpha,\beta} = k_{\beta}$. In fact, $h_{\alpha,\beta}$ is the transition map for the frames above. Rewriting this we see that $k_{\alpha} h_{\alpha,\beta} = h_{\alpha,\beta} k_{\beta}$. Hence, we can define a map by $f_{\alpha} \colon U_{\alpha} \times \C^n \to U_{\alpha} \times \C^n$ by $f_{\alpha}(z, \xi) = (z, k_{\alpha}(z) \xi)$. Then, for every $z \in U_{\alpha} \cap U_{\beta}$ and every $\xi \in \C^n$, we have that 
\[
(h_{\alpha,\beta} f_{\alpha})(z,\xi) = (z, h_{\alpha,\beta} k_{\alpha}(z) \xi) = (z, k_{\beta}(z) h_{\alpha,\beta} \xi) = (f_{\beta} h_{\alpha,\beta})(z, \xi).
\]
Therefore, $f$ defines an automorphism of $\cF$. This leads us to the following definition.

\begin{dfn}
Let $Z$ be a paracompact space and $\cF$ a Hermitian vector bundle of dimension $n$ with structure bundle $S_n$. A complete Pick bundle structure on $\cF$ is given by a covering $\fU = \{U_{\alpha}\}_{\alpha \in A}$ of $Z$, along with a section $f$ of $\Aut(\cF)$, such that on every element of $\fU$, the bundle $\cF$ is trivial and $f$ gives us a continuous family of complete Pick kernels on $n$ points. We call the triple $(\cF,k,\fU)$ a complete Pick bundle.
\end{dfn}

\begin{quest} \label{quest:universality}
Let $Z$ be a compact topological space and $(\cF,k, \fU)$ a complete Pick bundle of rank $n$ on $Z$. What are the conditions on $Z$ and $(\cF,k, \fU)$, such that there exist $d \in \N$ and a continuous map $\varphi \colon Z \to \bX_{d,n}^0$, such that $\cF$ is the pullback of $\cE_{d,n}$ and $k$ is the corresponding kernel?
\end{quest}

In general, one might be forced to consider $\B_{\infty}$ the ball of $\ell^2$ and $\pi_{\infty,n} \colon \B_{\infty}^{n,0} \to \bX_{\infty, n}^0$. The space $\bX_{\infty,n}^0$ is a subspace of the Fadell configuration space \cite{FadellNeuwirth}, which is the classifying space for $S_n$. However, the methods in the paper must be refined to deal with $\bP_{\infty,n}$ and $\cE_{\infty,n}$.

\section{Cowen-Douglas class operators from hypersurfaces} \label{sec:Cowen-Douglas}

Assume that $p \in \C[z_1,\ldots,z_d]$ is a reduced homogeneous polynomial of degree $n$. Let $Z(p)$ be the variety cut out by $p$ and $V(p) = \B_d \cap Z(p)$. We may assume (up to a unitary) that $p(0,\ldots,0, 1) \neq 0$. 

\begin{lem} \label{lem:lower_bound_coordinate}
There exists $c < 1$, such that for all $\lambda \in V(p)$, $|\lambda_d| \leq c$.
\end{lem}
\begin{proof}
Let $e_d = (0,\ldots,0,1)$ and $\overline{V(p)} = \overline{\B_d} \cap Z(p)$. Note that $\lambda_d = \langle \lambda, e_d \rangle$. This function is continuous on $\overline{V(p)}$ and bounded by $1$ by Cauchy-Schwartz. However, by our assumption, $p(e_d) \neq 0$. Thus, by the equality close of Cauchy-Schwartz and the fact that $p$ is homogenous, we get that $\lambda \mapsto |\lambda_d|$ never attains the value $1$ on $\overline{V(p)}$. By compactness, there exists $c < 1$, such that $\max_{\lambda \in \overline{V(p)}} |\lambda_d| = c$.
\end{proof}

We will write $z' = (z_1,\ldots,z_{d-1})$. We aim to produce a vector bundle of complete Pick spaces from projecting the hypersurface cut out by $p$ onto the coordinate plane of $z'$. The following lemma is necessary since we are dealing with the part of the hyperplane inside $\B_d$.

\begin{lem} \label{lem:bundle_neighborhood}
There exists a neighborhood $U$ of the origin in the coordinate plane of $z'$, such that for every $\lambda \in U$, $p(\lambda,z_d)$ has precisely $n$ roots counting multiplicities in $\B_d$.
\end{lem}
\begin{proof}
This follows from the continuity of roots of $p(\lambda,z)$ in $\lambda$ and the fact that $p(0,z) = a z^n$, for some $a \neq 0$.
\end{proof}
\begin{remark} \label{choice of U}
By shrinking $U$, if necessary, we may assume that $U$ is invariant under conjugation (say, a polydisc or a ball). From now on, we will assume that it is the case when we speak of $U$.
\end{remark}
Let $\cH_p = (p H^2_d)^{\perp}$. Note that $\cH_p = \oplus_{m=0}^{\infty} \cH_{p,m}$, where $\cH_{p,m}$ is the $m$-th graded component of $\cH_p$. We will denote by $P_m$ the projection onto $\cH_{p,m}$, for $m \in \N \cup \{0\}$. Let $S_j = P_{\cH_p} M_{z_j} |_{\cH_p}$ for $1 \leq j \leq d$. Let $S = (S_1,\ldots,S_d)$ be the corresponding row contraction. Recall that $S S^* = \sum_{j=1}^d S_j S_j^* = I - P_0$, where $P_0$ is the projection onto the graded component of degree $0$ in $\cH_p$.

Now let $U$ be as above and $\lambda \in U$. Set $\cH_{p,\lambda} = \cH_p \cap \bigcap_{j=1}^{d-1} \ker \left(M_{z_j} - \lambda_j I\right)^*$ and observe that
\[
\cH_{p, \lambda}^{\perp} = \overline{p H^2_d} + \sum_{j=1}^{d-1} (z_j - \lambda_j) H^2_d.
\]
Let $J_{p,\lambda} \subset \C[z_1,\ldots,z_d]$ be the ideal generated by $p$ and $z_j - \lambda_j$ for $j=1,\ldots,z_{d-1}$. Therefore, by what we have above, $\cH_{p,\lambda} = J_{p,\lambda}^{\perp}$. 

\begin{lem} \label{lem:fibers}
We have that $\dim \cH_{p,\lambda} = n$ and if $\mu_1,\ldots,\mu_r$ are the roots of $p(\lambda,z)$ of multiplicities $m_1,\ldots,m_r$, respectively, then $\cH_{p,\lambda}$ is spanned by $\{k_{(\lambda,\mu_j)}^{(\ell e_d)}\}_{j=1, \ell = 0}^{r, m_j-1}$, where $e_d$ is the last vector of the standard basis of $\C^d$.
\end{lem}
\begin{proof}
Fix $1 \leq j \leq k$ and $0 \leq \ell \leq m_j - 1$. For every $f \in H^2_d$ and $1 \leq i \leq d-1$, we have that
\[
\left\langle f, (M_{z_i} - \lambda_i)^* k_{(\lambda, \mu_j)}^{(\ell e_d)} \right\rangle = \left\langle (z_i - \lambda_i) f, k_{(\lambda, \mu_j)}^{(\ell e_d)} \right\rangle = \dfrac{\partial^{\ell}}{\partial z_d^{\ell}}\left((z_i - \lambda_i) f \right) (\lambda, \mu_j) = 0.
\]
Hence, $k_{(\lambda, \mu_j)}^{(\ell e_d)} \in \bigcap_{i=1}^{d-1} \ker (M_{z_i} - \lambda_i I)^*$. Now for $f \in H^2_d$, we have that
\[
\left\langle p f, k_{(\lambda, \mu_j)}^{(\ell e_d)} \right\rangle =\sum_{b = 0}^{\ell} \dfrac{\partial^{b} p}{\partial z_d^{b}}(\lambda,\mu) \dfrac{\partial^{\ell-b} f}{\partial z_d^{\ell-b}}(\lambda,\mu) = 0.
\]
The last equality follows from the fact that $\dfrac{\partial^{b} p}{\partial z_d^{b}}(\lambda,\mu) = 0$ for $0 \leq b \leq m_j-1$. Hence, $k_{(\lambda, \mu_j)}^{(\ell e_d)} \in \cH_{p,\lambda}$. Moreover, the vectors $k_{(\lambda, \mu_j)}^{(\ell e_d)}$ are linearly independent. Thus, $\dim \cH_{p,\lambda} \geq n$. 

On the other hand, the quotient $\C[z_1,\ldots,z_d]/J_{p,\lambda}$ is finite dimensional, since it is isomorphic to $\C[z]/(p(\lambda,z))$. The dimension of the latter space is $n$. Now we can choose a basis $q_1,\ldots,q_n$ for a complement of $J_{p,\lambda}$. Hence, we can write every $q \in \C[z_1,\ldots,z_d]$ as $q = q_0 + \sum_{j=1}^n \alpha_j q_j$, where $q_0 \in J_{p,\lambda}$. Hence, for every $f \in \cH_{p,\lambda}$,
\[
\langle q, f \rangle = \sum_{j=1}^d \alpha_j \langle q_j, f \rangle.
\]
In other words, $f$ is determined completely, by $\langle q_j, f \rangle$ for $j=1,\ldots,n$. We conclude that $\dim \cH_{p,\lambda} = n$.
\end{proof}

Therefore, we obtain a map $\varphi \colon U \to \bP_{d,n}$ given by $\varphi(\lambda) = P_{\cH_{p, \lambda}}$. To show that this map is continuous, we will show that the tuple of operators $S^{\prime *} = (S_1^*,\ldots,S_{d-1}^*)$ is in the Cowen-Douglas class for $U$. To do this, we first observe that it is clear that for $\lambda \in U$, $\ker(S^{\prime *} - \bar{\lambda} I) = \cH_{p, \lambda}$. Moreover, if $f \in \cH_p$ is orthogonal to $\bigvee_{\lambda \in U} \cH_{p, \lambda}$, then $f = 0$. This condition, in particular, implies that $f$ is identically $0$ on a neighborhood of $0$ in $V(p)$. Hence, by analyticity, it is identically $0$ on $V(p)$. Therefore, we only need to show that the range of the operator $S^{\prime *} - \overline{\lambda}I \colon \cH_p \to \cH_p \otimes \C^{d-1}$ is closed, for each $ \lambda \in U $.

\begin{lem}
Let $ \lambda \in U $, $ f \in \mathcal{ H }_p $ and $ \mu_1, \hdots, \mu_r \in \C $ be the roots of the one variable polynomial $ p ( \lambda, \cdot ) $ with multiplicities $m_1, \ldots, m_r$. Then, $ f $ takes the form
$$ f = \sum_{ j = 1 }^{ d - 1 } ( S_j - \lambda_j ) h_j $$
for some $ h_1, \hdots, h_{ d - 1 } \in \mathcal{ H }_p $ if and only if every point $(\lambda, \mu_j)$ is a root of $f$ of order $m_j$, for $1 \leq j \leq r$.
\end{lem}

\begin{proof}
If $f = \sum_{ j = 1 }^{ d - 1 } ( S_j - \lambda_j ) h_j $ for some $ h_1, \hdots, h_{ d - 1 } \in \mathcal{ H }_p $, then, clearly, $f \perp \cH_{p,\lambda}$, as desired. To prove the converse, we begin by showing that any $ f \in H^2_d $, such that every point $(\lambda, \mu_j)$ is a root of $f$ of order $m_j$, for $1 \leq j \leq r$. satisfies the following identity
\begin{equation} \label{factorization of f} f ( z ) = \sum_{ j = 1 }^{ d - 1 } ( z_j - \lambda_j ) g_j + \prod_{ i = 1 }^r ( z_d - \mu_i )^{m_i} g_d \end{equation}
for some $ g_1, \hdots, g_d \in H^2_d $. We prove this with the help of mathematical induction on $ n = \deg p = \sum_{i=1}^r m_i $. To this end, we write our roots as $\mu_1,\ldots, \mu_n$, repeating them with appropriate multiplicity. Since $ f ( \lambda, \mu_1 ) = 0 $, it follows from the solution to the Gleason problem in $ H^2_d $ \cite[Section 3]{AlpKapt-gleason}, that there exists $ f_1^1, \hdots, f_d^1 \in H^2_d $ such that
$$ f ( z ) = \sum_{ i = 1 }^{ d - 1 } ( z_i - \lambda_i ) f_i^1 + ( z_d - \mu_1 ) f_d^1 $$
verifying the desired identity for $ n = 1 $. Now assume that the Equation \eqref{factorization of f} holds for $ k < n $ observe that we have following two cases -- either $ \mu_{ k + 1 } = \mu_k $ or $ \mu_{ k + 1 } \neq \mu_k $.

\begin{description}
\item[Case I] Let $ \mu_{ k + 1 } \neq \mu_k $. Since $ f ( \lambda, \mu_{ k + 1 } ) = 0 $ and $ \mu_i \neq \mu_{ k + 1 } $ for $ 1 \leq i \leq k $, it follows from induction hypothesis that $ g_d ( \lambda, \mu_{ k + 1 } ) = 0 $. Consequently, Gleason property of $ H^2_d $ yields that there exist $ g_j^1, \hdots, g_d^1 $ such that
$$ g_d ( z ) = \sum_{ j = 1 }^{ d - 1 } ( z_j - \lambda_j ) g_j^1 + ( z_d - \mu_{ k + 1 } ) g_d^1 . $$ Now substituting $ g_d $ in the  Equation \eqref{factorization of f} with $ k $ in place of $ n $, we have that 
$$ f ( z ) = \sum_{ j = 1 }^{ d - 1 } ( z_j - \lambda_j ) \left( g_j +  \prod_{ i = 1 }^k ( z_d - \mu_i ) g_j^1 \right)  + \prod_{ i = 1 }^{ k + 1 } ( z_d - \mu_i ) g^1_d  $$
verifying the induction step.

\item[Case II] Let $ \mu_{ k + 1 } = \mu_k $ and $ t $ be the cardinality of the set $ \{ i : \mu_{ k + 1 } = \mu_i, i = 1, \hdots, k \} $. Note that $ \frac{ \partial^t f}{ \partial z_d^t } ( \lambda, \mu_{ k + 1 } ) = 0 $. Consequently, differentiating the equation 
$$ f ( z ) = \sum_{ j = 1 }^{ d - 1 } ( z_j - \lambda_j ) g_j + \prod_{ i = 1 }^k ( z_d - \mu_i ) g_d  $$
with respect to $ z_d $ of order $ t $ and evaluating at $ ( \lambda, \mu_{ k +1 } ) $ we have that  $ g_d ( \lambda, \mu_{ k + 1 } ) = 0 $. Then, the induction step, in this case, follows from a similar argument as in Case I. 
\end{description}

Now observe that 
\begin{eqnarray*}
f ( z ) & = &  \sum_{ j = 1 }^{ d - 1 } ( z_j - \lambda_j ) g_j + \prod_{ i = 1 }^n ( z_d - \mu_i ) g_d \\
%& = & \sum_{ j = 1 }^{ d - 1 } ( z_j - \lambda_j ) g_j + \prod_{ i = 1 }^n ( z_d - \mu_i ) g_d - p g_d +p g_d \\
& = & \sum_{ j = 1 }^{ d - 1 } ( z_j - \lambda_j ) g_j + \left(\prod_{ i = 1 }^n ( z_d - \mu_i) - p \right) g_d +p g_d \\
& = & \sum_{ j = 1 }^{ d - 1 } ( z_j - \lambda_j ) \widetilde{g}_j + p g_d. 
\end{eqnarray*}
Where the last equality holds since $ p $ is a monic polynomial of degree $ n $ and $ \mu_1, \hdots, \mu_r $ are the roots of $ p ( \lambda, \cdot ) $. Thus, $\prod_{ i = 1 }^n ( z_d - \mu_i) - p$ belongs to the ideal in the polynomial ring generated by $z_1-\lambda_1,\ldots,z_{d-1} - \lambda_{d-1}$. Furthermore, since $\cH_p$ is multiplier coinvariant, we have that $P_{ \mathcal{ H }_p } M_{z_j} P_{ \mathcal{ H }_p } = P_{ \mathcal{ H }_p } M_{z_j}$, for all $1 \leq j \leq d$. Therefore,
$$ f = P_{ \mathcal{ H }_p } ( f ) = \sum_{ j = 1 }^{ d -1 } ( S_j - \lambda_j ) P_{ \mathcal{ H }_p } ( g_j ) $$
completing the proof.
\end{proof}

\begin{thm} \label{thm:CD}
The tuple $S^{\prime *}$ is in the Cowen-Douglas class.
\end{thm}
\begin{proof}
The result follows from the Lemma above and \cite[Theorem 2, pp 756]{ChenDouglas}.
\end{proof}

\begin{example}
Let us consider the quotient Hilbert space space $ \mathcal{ H }_p $ corresponding to the polynomial $ p ( z_1, z_2 ) = ( z_1 - \lambda z_2 ) ( z_1 - \mu z_2 ) $ as in the theorem above. Recall from Theorem \ref{thm:CD} that the operator $ S_2^* $ is in the Cowen-Douglas $ \mathrm B_2 ( U ) $ over the open disc $ U \subset \{ 0 \} \times \mathbb{ D } $ as specified in Lemma \ref{lem:bundle_neighborhood}. Let $ E_{ S_2 } \rightarrow U $ be the hermitian holomorphic vector bundle associated with $ S^*_2 $. This is a vector bundle of rank $ 2 $ with $ \left\{ \varphi_1(z) =  k_{ ( \lambda z, z ) }, \varphi_2(z) = \frac{ k_{ ( \lambda z, z ) } - k_{ (\mu z, z) } }{ \bar{ z } } \right\} $ as a holomorphic frame over $ ( 0, z ) \in U \setminus\{0\}$. One can show, as in the proof of Lemma \ref{lem:derivatives}, that the frame can be extended to $0$ by assigning $\varphi_1(0) = 1$ and $\varphi_2(0) = \overline{(\lambda - \mu)} z_1$.  The hermitian metric with respect to this frame then turns out to be 
\begin{equation} \label{hermitian metric}  
H ( z ) = \begin{pmatrix}
    \frac{ 1 }{ ( 1 - a | z |^2 ) } & \frac{ \lambda ( \overline{ \lambda - \mu } ) \overline{ z } }{ ( 1 - a | z |^2 )( 1 - c | z |^2 ) } \\
    \frac{ \overline{ \lambda } ( \lambda - \mu  ) z }{ ( 1 - a | z |^2 )( 1 - \overline{ c } | z |^2 ) } & \frac{ | \lambda - \mu |^2 ( 1 - ( 1 - | \lambda| | \mu | ) | z |^2 )( 1 - ( 1 + | \lambda | | \mu | ) | z |^2 ) }{ ( 1 - a | z |^2 ) ( 1 - b | z |^2 ) | 1 - c | z |^2 |^2 }
\end{pmatrix} 
\end{equation}
Here, $ a = 1 + | \lambda |^2 $, $ b = 1 + | \mu |^2 $ and $ c = 1 + \lambda \overline{ \mu } $. Consequently, the hermitian metric $ h $ of the determinant bundle $ \bigwedge^2 E_{ S_2 } \rightarrow U $ corresponding to the vector bundle $ E_{ S_2 } \rightarrow U $ is by definition
$$ h ( z ) = \det H ( z ) = \frac{ | \lambda - \mu |^2 ( 1 - | z |^2 ) }{ ( 1 - a | z |^2 ) ( 1 - b | z |^2 ) | 1 - c | z |^2 |^2 }, ~~~ z \in U . $$
The advantage of the determinant bundle is that it is a line bundle, and hence, the curvature of this bundle is a complete invariant. We now compute the curvature $ \mathcal{ K }_{ \bigwedge^2 E_{ S_2 } } $ of this line bundle at $ z = 0 $, as follows:
$$ \mathcal{ K }_{ \bigwedge^2 E_{ S_2 } } ( 0 ) = \partial \overline{ \partial } \log ( h ( z ) )|_{ z = 0 } = 3 + | \lambda + \mu |^2 . $$
\end{example}

In \cite{LuYangZu}, the authors study the irreducibility of compressed shifts on the bidisc. Here, we treat a special case of irreducibility when $d = n = 2$. The lemma below, modelled after Lemma 5.1 in \cite{KorMisra}, can be proved precisely in the same way as in the original proof, so it is omitted.

\begin{lem} \label{normalized kernel}
Suppose that the adjoint of the multiplication operators $M_z$ on a reproducing kernel Hilbert space $\cH$ with the reproducing kernel $K$ is in $\mathrm B_r(\Omega)$ for some open connected domain $\Omega \subset \C $. If there exists an orthogonal projection $X$ commuting with the operator $M_z$, then 
$$\Phi_X(z)K(z,w)~=~K(z,w)\Phi_X(w)^{*}$$ for some holomorphic function $\Phi_X:\Omega \rightarrow \C^{r\times r}$ with $\Phi_X^2=\Phi_X$.
\end{lem}

In our set-up for $ p ( z_1, z_2 ) = ( z_1 - \lambda z_2 ) ( z_1 - \mu z_2 ) $, recall that $ S_2^* $ on $ \mathcal{ H }_p $ is in $ \mathrm B_2 ( U ) $. Consequently, it is well known \cite{CowenDouglas=CGOT, CurtoSalinas} that $ S_2 $ can be modelled as the multiplication operator by coordinate functions on a reproducing kernel Hilbert space of holomorphic functions on $ U $ (see Remark \ref{choice of U}). One such reproducing kernel $ K : U \times U \rightarrow \text{M}_2 ( \C ) $ can be defined as
\begin{equation} \label{rk for S_2}
K ( z, w ) : = \begin{pmatrix}
    \frac{ 1 }{ ( 1 - a z \overline{w} ) } & \frac{ \lambda ( \overline{ \lambda - \mu } ) \overline{ w } }{ ( 1 - a z \overline{ w } )( 1 - c z \overline{ w } ) } \\
    \frac{ \overline{ \lambda } ( \lambda - \mu  ) z }{ ( 1 - a z \overline{ w } )( 1 - \overline{ c } z \overline{ w } ) } & \frac{ | \lambda - \mu |^2 ( 1 - ( 1 - | \lambda| | \mu | ) z \overline{ w } )( 1 - ( 1 + | \lambda | | \mu | ) z \overline{ w } ) }{ ( 1 - a z \overline{ w } ) ( 1 - b z \overline{ w } ) ( 1 - c z ) ( 1 - \overline{ c } \overline{ w } ) }
\end{pmatrix}
\end{equation}
Thus, $S_2^*$ on $\cH_p$ is irreducible if and only if there is no non-trivial projection $X_0$ on $\C^2$ satisfying the following equation.
$$
X_0K_0(z,0)^{-1}K_0(z,w)K_0(0,w)^{-1}~=~K_0(z,0)^{-1}K_0(z,w)K_0(0,w)^{-1}X_0, ~~~ z, w \in U. 
$$
Here, $K_0(z,w)=K(0,0)^{-\frac{1}{2}}K(z,w)K(0,0)^{-\frac{1}{2}}$ and $K$ is the reproducing kernel described in the equation \eqref{rk for S_2}. Let $\hat{K}_0(z,w)$ -- called the normalized kernel -- denote the kernel $K_0(z,0)^{-1}K_0(z,w)K_0(0,w)^{-1}$ on $ U $.

 \begin{thm}
 Let $ \lambda, \mu \in \C $ and $ p $ be the polynomial $ p ( z_1, z_2 ) = ( z_1 - \lambda z_2 ) ( z_1 - \mu z_2 ) $. Then the operator $ S_2^* $ on $ \cH_p $ is irreducible.
 \end{thm}
 
 \begin{proof}
 Since $\hat{K}_0$ is equivalent to the reproducing kernel $K$ defined in the equation \eqref{rk for S_2}, it is enough to show that the multiplication operator on the reproducing kernel Hilbert space corresponding to the kernel $\hat{K}_0$ by the coordinate function is irreducible. In view of Lemma \ref{normalized kernel} above, it amounts to show that $I_{2 \times 2}$ is the only self-adjoint projection matrix commuting with $\hat{K}_0(z, w)$ for all $z, w \in U $ where $I_{2 \times 2 }$ is the identity matrix of size $2$. So let us begin with a $ 2 \times 2 $ self-adjoint projection matrix $X$ satisfying
 $$ X \hat{K}_0 ( z, w ) = \hat{K}_0(z, w ) X, ~~~ \text{for all} ~~ z,w \in U . $$
In particular, $X$ commutes with the matrix $\partial \overline{\partial}\hat{K}_0 (0,0)$ which turns out to be a diagonal matrix with diagonal entries $1$ and $2+|\lambda +\mu|^2$. Consequently, $X$ has to be a diagonal matrix. Moreover, since the $12$-th entry of $ \hat{K}_0 $ turns out to be 
$$ \hat{K}_0 ( z, w )_{ 12 } = \frac{ \lambda | \lambda - \mu | ( 1 + \lambda \overline{\mu}) z \overline{w}^2 }{ ( \lambda - \mu ) ( 1 - ( 1 + |\lambda|^2) z \overline{w})( 1 - ( 1 + \lambda \overline{ \mu } ) z\overline{w} )}, $$ there exists $z_0 \in U $ such that the $12$-th entry of $\hat{K}_0(z_0,z_0)$ is not zero. However, this implies that the diagonal entries of $X$ have to be equal as $X$ also commutes with $\hat{K}_0(z_0,z_0)$. Finally, since $X$ is a projection, $X$ has to be $ I_{2 \times 2}$.
\end{proof}

\section{Stratification of $\pick_n$} \label{sec:stratification}

Let $G = \Aut(\B_d)$. Recall that $G$ acts on $\bP_{d,n}$ by conjugation. Moreover, the bundle $\cE_{d,n}$ is equivariant with respect to the natural action on the total space of the bundle. We will identify $\bP_{d,n}^0$ with $\bX_{d,n}^0$ and consider $\cE_{d,n}$ as bundle on $\bX_{d,n}^0$. There is a natural stratification on $\bX_{d,n}$ given by
\[
\bX_{d,n}^0(m) = \left\{ [X] \in \bX_{d,n}^0 \mid \affspan(X) \leq m \right\}, \text{ for } 1 \leq m \leq \min\{d, n-1\}.
\]
For every $[X] \in \bX_{d,n}^0$ and every representative $X = \{\lambda_1,\ldots,\lambda_n\}$, we can consider the matrix with columns $C_X = \lambda_1 - \lambda_n, \ldots, \lambda_{n-1} - \lambda_n$. We note that $[X] \in \bX_{d,n}^0(m)$ if and only if the minors of size $m+1$ of $C_X$ all vanish. Therefore, $\bX_{d,n}^0 = \bX_{d,n}^0(\min\{d, n-1\})$ and for $m < \min\{d, n-1\}$, the sets $\bX_{d,n}^0(m)$ are closed. By \cite[Proposition 2.4.1]{Rudin-book}, the automorphisms of $\B_d$ preserve affine subsets of $\B_d$. Hence, the sets $\bX_{d,n}^0(m)$ are $G$-invariant. In particular, if we identify $\pick_n$ with $\bX_{n-1,n}^0/G$, we obtain a stratification of $\pick_n$ by the sets $\pick_n(m)$, which are the images of $\bX_{n-1,n}^0(m)$ under the quotient map.

There is another way to describe $\pick_n(m)$. Let $\cH$ be a complete Pick space on the set $X = \{x_1,\ldots,x_n\}$. We may assume that the kernel $k$ of $\cH$ is normalized at $x_1$. This implies that the function $g(x_i,x_j) = 1 - 1/k(x_i,x_j)$ is a positive kernel. We will call the dimension of the corresponding reproducing kernel Hilbert space the embedding dimension of $\cH$ and denote it by $\edim(\cH)$. This is due to the fact that a Kolmogorov decomposition of $g$ gives an embedding of $X$ into $\B_{\edim(\cH)}$. It is clear that the embedding dimension is an invariant of the complete Pick space. We get that
\[
\pick_{n}(m) = \left\{ \cH \in \pick_n \mid \edim(\cH) \leq m \right\}.
\]
Indeed, one inclusion is clear since if we have $[X] \in \bX_{n-1,n}^0(m)$ for $m < n-1$, we move $[X]$ by an automorphism so that one of the points is the origin. We then note that $[X]$ is contained in a subspace, and the intersection of this subspace with $\B_{n-1}$ is a unit ball of dimension at most $m$. Hence, the image of $[X]$ in $\pick_n$ lies in $\pick_n(m)$. Conversely, for $\cH$ as above with $\edim(\cH) \leq m$, we take the map $b \colon X \to \B_d$, for $d \leq m$ and compose it with the embedding of $\B_d$ in $\B_{n-1}$ as the intersection with a coordinate plane. This shows that equivalence class of $\cH$ in $\pick_n$ lies in the image of $\bX_{n-1,n}^0(m)$.

Also, the argument above shows that for points $[X] \in \bX_{d,n}^0(m)$, we have a non-trivial stabilizer. In fact, the stabilizer is a compact infinite subgroup of $G$. First, consider the open dense submanifold $\bX_{n-1,n}^0 \setminus \bX_{n-2,n-1}^0 $ of $\bX_{n-1,n}^0$. Observe that there are points in this submanifold with a non-trivial stabilizer. To see this, fix $ 0 < r < 1$ and consider the point $[X] = [0] + \sum_{j=1}^{n-1} r [e_j]$. Clearly, the standard representation of $S_{n-1}$ stabilizes this point. However, if $[X]$ is such that the pseudohyperbolic distances between every two points in $[X]$ are different, then since $G$ acts by isometries on $\B_d^{n,0}$, the stabilizer of $[X]$ is trivial. It is immediate from the definition of the topology on $\bX_{d,n}$ that the set of all such points is open and dense in $\bX_{d,n}^0$. Moreover, all those points are regular points in the sense of \cite[Definition 2.8.3]{DuisKolk} or \cite[Definition on page 12]{Koszul-book} (see also \cite[Theorem 2]{Koszul-book}). Let us denote by $\bX_{n-1,n}^{0,reg}$ the collection of all regular points and by $\pick_n^{reg}$ its image in $\pick_n$. By \cite[Theorem 4]{Koszul-book}, $\pick_n^{reg}$ is connected and every point in $\bX_{n-1,n}^{0,reg}$ has trivial stabilizer. We summarize this discussion and its consequences in the following theorem.

\begin{thm} \label{thm:reg_pick_spaces}
The space $\pick_n^{reg}$ has a structure of a smooth manifold and $\bX_{n-1,n}^{0,reg}$ is the total space of a principal $G$-bundle over $\pick_n^{reg}$. Moreover, the vector bundle $\cE_{n-1,n}|_{\bX_{n-1,n}^{0,reg}}$ gives rise to a vector bundle $\cE_n^{reg}$ on $\pick_n^{reg}$.
\end{thm}
\begin{proof}
The first part follows from \cite[Proposition 5.2]{tomDieck-book}. The second follows from \cite[Proposition 9.4]{tomDieck-book}
\end{proof}
We note that the stratification above is part of the orbit type stratification on $\bX_{d,n}^0$.

\bibliographystyle{abbrv}
\bibliography{bibliography}
\end{document}